\documentclass[11pt]{article}
\usepackage{amsmath,amssymb,amsfonts,amsthm}
\usepackage[english]{babel}

\newcommand{\cA}{{\mathcal A}}
\newcommand{\cB}{{\mathcal B}}
\newcommand{\cD}{{\mathcal D}}
\newcommand{\cE}{{\mathcal E}}
\newcommand{\cF}{{\mathcal F}}

\newcommand{\cH}{{\mathcal H}}
\newcommand{\cK}{{\mathcal K}}

\newcommand{\cM}{{\mathcal M}}
\newcommand{\cN}{{\mathcal N}}
\newcommand{\cO}{{\mathcal O}}

\newcommand{\cW}{{\mathcal W}}

\newcommand{\al}{{\alpha}}
\newcommand{\om}{{\omega}}
\newcommand{\Om}{{\Omega}}
\newcommand{\Ga}{{\Gamma}}
\newcommand{\ga}{{\gamma}}
\newcommand{\ka}{{\kappa}}

\newcommand{\si}{{\sigma}}

\newcommand{\la}{\lambda}
\newcommand{\La}{\Lambda}
\newcommand{\eps}{\varepsilon}

\newcommand{\R}{{\mathbb R}}

\newcommand{\Mg}{{ M_\gamma}}

\newcommand{\mfA}{{\mathfrak{A}}}

\newcommand{\myspace}{\qquad\qquad\qquad}
\newcommand{\hf}{\frac12}

\newcommand{\rarrow}{\rightarrow}

\newcommand{\g}{\nabla}
\newcommand{\wrt}{with respect to}

\newcommand{\be}{\begin{equation}}
\newcommand{\ee}{\end{equation}}
\newcommand{\ba}{\begin{array}}
\newcommand{\ea}{\end{array}}
\newcommand{\bey}{\begin{eqnarray}}
\newcommand{\eey}{\end{eqnarray}}
\newcommand{\beyy}{\begin{eqnarray*}}
\newcommand{\eeyy}{\end{eqnarray*}}

\newtheorem{theorem}{Theorem}[section]
\newtheorem{lemma}[theorem]{Lemma}
\newtheorem{proposition}[theorem]{Proposition}
\newtheorem{remark}[theorem]{Remark}

\newtheorem{assumption}[theorem]{Assumption}

\newtheorem{definition}[theorem]{Definition}
\newtheorem{corollary}[theorem]{Corollary}
\numberwithin{equation}{section}

\begin{document}

\title{Long-time dynamics of a coupled system of nonlinear wave and 
thermoelastic plate equations\footnote{This research was started while the second author was visiting the Dipartimento di Matematica Applicata, 
Universit\`a degli Studi di Firenze, within the program ``Permanenza presso le Unit\`a Amministrative di studiosi stranieri {\em di chiara fama}''
of the International Relations Office of the same University.
The research of the first author was partially supported
by the Italian Ministero dell'Istruzione, dell'Universit\`a e della Ricerca,
within the (2005) PRIN Project ``Regularity, qualitative properties
and control of solutions to nonlinear partial differential equations''.}
}

\author{\sc Francesca Bucci\\[1mm]
\footnotesize Universit\`a degli Studi di Firenze,
Dipartimento di Matematica Applicata\\
\footnotesize Via S.~Marta 3, 50139 Firenze, Italy\\
\footnotesize E-mail: {\tt francesca.bucci@unifi.it}
\\[3mm]
\sc Igor Chueshov\\[1mm]
\footnotesize Kharkov National University,
Department of Mathematics and Mechanics\\
\footnotesize 4 Svobody sq, 61077 Kharkov, Ukraine\\
\footnotesize E-mail: {\tt chueshov@univer.kharkov.ua}
}

\date{}

\maketitle

{\footnotesize
\begin{abstract}
\noindent
We prove the existence of a compact, finite dimensional, global attractor
for a coupled PDE system comprising a nonlinearly damped semilinear wave
equation and a nonlinear system of thermoelastic plate equations,
without any mechanical (viscous or structural) dissipation in the
plate component.
The plate dynamics is modelled following Berger's approach; we investigate
both cases when rotational inertia is included into the model and when it 
is not.
A major part in the proof is played by an estimate---known as 
stabilizability estimate---which shows that the difference of any 
two trajectories can be exponentially stabilized to zero, modulo a 
compact perturbation.
In particular, this inequality yields bounds for the attractor's fractal 
dimension which are independent of two key parameters, namely $\gamma$ 
and $\kappa$, the former related to the presence of rotational inertia 
in the plate model and the latter to the coupling terms.
Finally, we show the upper semi-continuity of the attractor with 
respect to these parameters.
\end{abstract}



}


\section{Introduction}
A {\em structural acoustic} interaction can be modelled by a coupled
system of partial differential equations (PDE) comprising (i) an equation
describing acoustic waves in a given two or three dimensional chamber
and (ii) an elastic (beam, plate, or shell) equation describing
the vibrations of a (flexible) part of the chamber's wall;
see, e.g., \cite{morse} or \cite{howe}.
The interaction takes place on the boundary between the acoustic medium and
the structure. The problem is thus described by a mathematical model which
couples two (or more) partial differential equations, possibly of different
character, the former defined on an $n$-dimensional manifold and the latter
on an $(n-1)$-dimensional one, respectively.

Spurred on by engineering applications in aerospace industry, linear and
nonlinear structural acoustic models have recently received novel attention,
bringing about new developments (mainly, but not only) in the field of 
mathematical control theory.
Various theoretical issues pertaining to these composite PDE systems have
been explored so far, such as well-posedness, interior and boundary regularity,
controllability, stabilization, and optimization; see, e.g., the monograph \cite{las-cmbs}.
\par
Because of their physical motivation, structural acoustic models which
include {\em thermal effects} in the vibrating wall have been the subject
of recent research as well.
In particular, the study of the uniform stability properties of structural
acoustic models with thermoelastic wall in the case of a single equilibrium
may be found in the papers \cite{lale99,lale00,lale02,le00};
see also \cite{ryz04} for the case when the chamber is the half-space  
$\R^3_+$.
\\
In this paper, we consider a {\em nonlinear} structural acoustic model
which includes thermal effects and {\em does not contain} any mechanical 
dissipation in the plate component.
Accordingly, the PDE system displays an additional coupling of the elastic
equation with a heat equation.
The precise mathematical description of the problem is given below;
see \eqref{pde-system}.
It is our aim to investigate the long-term behaviour ($t\to \infty$) of
problem \eqref{pde-system}.
More precisely, the following questions are addressed: (i) well-posedness,
i.e. existence and uniqueness of the solution as well as continuous dependence
on the initial data; (ii) existence of a compact global attractor and its 
structure; (iii) finite dimensionality and other properties of the attractor.
\par
We recall that depending on whether rotational moments are included
into the model, the thermal component brings about a 
`predominantly hyperbolic' or parabolic character in the overall 
thermoelastic system. 
The present analysis takes in both cases (which correspond to either
$\gamma>0$ or $\gamma=0$ in the elastic equation of system 
\eqref{pde-system} below, respectively).
In addition, it will be shown that the size and the dimension of the 
global attractor are uniformly bounded with respect to the parameter $\gamma$, 
as well as to the `coupling' parameter $\kappa$ 
(where $\kappa=0$ corresponds to the case of the uncoupled wave equation
and thermoelastic system).
\par
Let us make some preliminary historical and bibliographical remarks.
Bounds for the size of the attractor which are uniform with respect to 
the parameter $\gamma$ (rotational inertia) were first obtained in 
\cite{Chu92} for a linearly damped von Karman plates equation.
Later, uniform bounds for the size---and also the dimension---of 
the attractor were establised in \cite{chu-las-preprint06} 
for a thermoelastic von Karman model, in the absence 
of mechanical dissipation.
Inspired by the work of \cite{chu-las-preprint06}, we aim to 
establish similar uniformity properties for the attractor of
the PDE system under investigation, still without including 
any type of dissipation in the plate component.
As we shall see later, in doing so we adapt some techniques 
developed in \cite{chu-las-preprint06}.
It should be observed that the complex scenario described 
above, that is (i) no mechanical dissipation is present, and 
(ii) the cases $\gamma=0$ and $\gamma>0$ 
({\em analytic vs non analytic} case) are treated at once, 
was first considered by the authors of \cite{avalos-las-2} 
in the study of the uniform stability problem for a (linear) 
system of thermoelasticity. (See also \cite{avalos-las-3} and \cite{avalos-1a}.)
\par
Before introducing the PDE model and giving a more detailed description of
the paper's contributions, some further comments are in order.
The theory of infinite-dimensional dynamical systems
is rich and beautiful and has extensive applications.
It is of course difficult to do justice to the richness of the subject,
and we shall just mention the very good references---for both the general
theory and the many applications---\cite{babin} and \cite{Tem97};
see also \cite{chueshov-book}, \cite{hale}.
It should however be observed that the results achieved in the present paper
benefit also from methods and tools pertaining to control theory, which only
very recently gave rise to significant progress in the asymptotic analysis
of hyperbolic (hyperbolic-like) dynamics with nonlinear dissipation; see 
\cite{mem}.
In this paper, we bring forward some new techniques 
developed in \cite{mem} for second order evolutionary equations
and in \cite{chu-las-preprint06} for thermoelastic plate equations.
The technical steps will include, in particular: (i) proving a basic
inequality which allow to reconstruct the energy of the system from 
the thermal and wave dissipation (see Proposition~\ref{pr:main}), 
and (ii) deriving a subsequent key estimate (Proposition~\ref{p:stabiliz}), 
referred to in the recent literature as {\em stabilizability inequality}, 
since it is reminiscent of the ones that occurr in uniform stabilization 
problems.
This kind of estimates has been originally derived in different
contexts for the study of dissipative wave dynamics;
see \cite{cel2,chu-las-1,mem,snowbird}.
(See also \cite{chu-las-jde06} for von Karman plates, \cite{bcl-cpaa-06}
for a composite wave-plate system and 
\cite{chu-las-preprint06} for the thermoelastic von Karman model).
\par
The PDE system under investigation is described as follows.
Let $\Om\subset\R^n$ be an open bounded domain, $n=2,3$,
with boundary $\partial\Om =:\Gamma=\overline{\Gamma_0\cup\Gamma_1}$
comprising two open (in the induced topology), connected, disjoint
parts $\Ga_0$ and $\Ga_1$ of positive measure.
We assume that either $\Omega$ is sufficiently smooth (e.g., $\Gamma\in C^2$)
or else $\Omega$ is convex.
$\Ga_0$ is \emph{flat} and is referred to as the {\em elastic wall},
whose dynamics is described by a thermoelastic Berger plate ($n=3$)
or beam ($n=2$) equation;
for details on Berger model in the isothermal case we refer
to \cite[Chap.~4]{chueshov-book} and to the literature quoted therein.
The acoustic medium in the chamber $\Om$ is described  by a semilinear
wave equation in the variable $z$, while $v$ denotes the vertical displacement
of the plate (or beam).
Then, we consider the following coupled PDE system
\begin{equation}\label{pde-system}
\left\{
\ba{lll}
z_{tt}+  g(z_t)  -  \Delta z + f(z) = 0
& {\rm in} & \Omega\times (0,T)
\\[2mm]
\displaystyle \frac{\partial z}{\partial\nu}  = 0
& {\rm on} & \Gamma_1\times (0,T)
\\[2mm]
\displaystyle\frac{\partial z}{\partial\nu} = \alpha\kappa  \,v_t
& {\rm on} & \Gamma_0\times(0,T)
\\[2mm]
v_{tt} -\ga \Delta v_{tt}   +\Delta^2 v
+ \Big[Q -\displaystyle\int_{\Gamma_0}|\nabla v(x,t)|^2 dx\Big]
\,\Delta v
\\[2mm]{}\,{}\qquad{}\qquad{}\qquad {}\qquad {}\qquad {}\qquad
 + \beta \kappa z_t|_{\Gamma_0} +\Delta \theta= p_0
& {\rm in} & \Gamma_0\times(0,T)
\\[2mm]
v=\Delta v = 0 & {\rm on} & \partial\Gamma_0\times(0,T)
\\[2mm]
 \theta_{t} - \Delta \theta  -\Delta v_t=0
& {\rm in} & \Gamma_0\times(0,T)
\\[2mm]
\theta= 0 & {\rm on} & \partial\Gamma_0\times(0,T)\,,
\ea
\right.
\end{equation}
supplemented with initial data
\begin{equation}\label{pde-i-d}
\left\{
\ba{lll}
z(0,\cdot) = z_0\,, \quad z_t(0,\cdot) = z_1 & {\rm in} & \Omega
\\[2mm]
v(0,\cdot) = v_0\,, \quad v_t(0,\cdot) = v_1 & {\rm in} & \Gamma_0
\\[2mm]
\theta(0,\cdot) = \theta_0 & {\rm in} & \Gamma_0\,.
\ea
\right.
\end{equation}
In the above system, $g(s)$ is a non-decreasing function describing
the dissipation which may affect the wave component of the system,
while the term $f(z)$ represents a nonlinear force;
$\nu$ is the outer normal vector, $\al$ and $\beta$ are positive constants;
the parameter $0\le\kappa\le 1$ has been introduced in order to cover
also the case of non-interacting wave and plate equations ($\kappa =0$).
The boundary term $\beta \kappa z_t|_{\Gamma_0} $ describes the acoustic
pressure.
The real parameter $Q$ describes in-plane forces applied to the plate, while
$p_0\in L_2(\Om)$ represents transversal forces.
The parameter $\gamma\in [0,1]$ describes the rotational inertia
of the plate filaments.

We begin our discussion of the long-time behaviour of the
initial/boundary value problem \eqref{pde-system}--\eqref{pde-i-d}
with the preliminary study of its well-posedness.
We shall see that well-posedness follows by general results pertaining
to differential equations
driven by (locally) Lipschitz perturbations of maximal monotone operators.

Our first main result, Theorem~\ref{t:main}, states the existence of
a global attractor for problem \eqref{pde-system} under rather general
conditions on the nonlinear functions $g$ and $f$; see Assumption~\ref{hypo-1}.
Since the dynamical system is gradient, the main issue to be explored
is the asymptotic compactness of the corresponding semi-flow.
In turn, to show this property we use an idea due to Khanmamedov~\cite{khan}
in the form suggested in \cite{chu-las-jde06} (see also \cite[Sect.~2.1]{mem}
and Proposition~\ref{prop:khan} below).

The subsequent main result, Theorem~\ref{t:main2}, concerns the dimension
and smoothness of the global attractor.
It requires additional hypotheses on the growth of the damping function $g$
and the nonlinear force $f$; see Assumption~\ref{hypo-2}.
In view of the new ideas and techniques developed in \cite{mem}, the principal
part of Theorem~\ref{t:main2}'s proof is the aforementioned 
{\em stabilizability estimate}; see Proposition~\ref{p:stabiliz}.
Indeed, this kind of inequalities---which are obtained by using energy methods
and are inspired by the ones used to show uniform stability of linear
and nonlinear PDE problems---have become a fundamental tool in the study of
finite dimensionality and smoothness of global attractors
(see, e.g., \cite{chu-las-1}, \cite{chu-las-preprint06}
and \cite{bcl-cpaa-06}).
It should be emphasized that not only these inequalities do not
follow from any general abstract result, but their derivation
strongly depends on the specific model under investigation.
In particular, achieving them in the present case requires a non trivial
modification of the techniques used in the isothermal case (\cite{bcl-cpaa-06})
as well as in the case of a thermoelastic von Karman plate model
(\cite{chu-las-preprint06}). This is accomplished in Lemma~\ref{le:f12}.

It is worth observing that so far global dynamics (i.e.~with
non-trivial attractors) of nonlinear structural acoustic models with multiple
equilibria as \eqref{pde-system} have been studied only in the isothermal case;
see \cite{bcl-cpaa-06}.
We also note that in the case $\kappa=0$---when the wave equation
and the thermoelastic system decouple---our main results yield
Corollary~3.7 in \cite{bcl-cpaa-06} for the wave component and
new results for the thermoelastic Berger model.

Finally, it is important to emphasize that the stabilizability estimate
shown in Proposition~\ref{p:stabiliz} is uniform with respect
to the parameters $\gamma,\kappa\in [0,1]$.
This enables us to prove that the size and the dimension of the attractor
admit bounds which are independent of $\gamma$ and $\kappa$.
This in turn naturally yields a third result (namely, Theorem~\ref{co:up-sc})
showing upper semi-continuity of the attractor with respect to the parameters
$\gamma$ and $\kappa$.
In particular, this semi-continuity property holds in the singular limit
$\gamma\to 0$, when the thermoelastic component of \eqref{pde-system}
changes character. See Remark~\ref{re:up-sc1} for more details.
For a general discussion about the continuity properties of attractors
with respect to parameters, see \cite{raquel2} and its references.

\smallskip
We conclude this Introduction by summarizing the contents of the paper's
various sections.
In Section~\ref{s:prelim} we introduce the essential elements of the
abstract set-up for problem \eqref{pde-system}, and we mainly discuss
its well-posedness. The proof of the corresponding result, that is
Theorem~\ref{thm:well-posed}, is outlined in the Appendix.
Section~\ref{s:statement} contains the statement of all our main results,
namely (i) Theorem~\ref{t:main} which asserts the existence of a compact
global attractor for system \eqref{pde-system}, then
(ii) Theorem~\ref{t:main2}, which describes the attractor's dimension and
smoothness, along with their Corollaries~\ref{c:wave} and \ref{c:plate},
and finally (iii) Theorem~\ref{co:up-sc} concerning the upper semi-continuity
of the attractor with respect to the parameters $\gamma$ and $\kappa$.
In Section~\ref{s:main-ineq} we show a preliminary inequality
which constitutes a first step for the proofs of both
Theorems~\ref{t:main} and~\ref{t:main2}.
Section~\ref{s:as-smooth} is mainly concerned with the asymptotic compactness
of the dynamical system generated by \eqref{pde-system}, as this is the core
of the proof of Theorem~\ref{t:main}.
In Section~\ref{s:stab-est} we derive the soughtafter stabilizability
estimate, in view of the novel and much challenging inequalities
established in Lemma~\ref{le:f12}.
The proofs of Theorem~\ref{t:main2} and Theorem~\ref{co:up-sc}
are finally given in Section~\ref{s:proof-t2}.


\section{Preliminaries}\label{s:prelim}
The notation below is largely standard within the literature.
For the reader's convenience, we just recall that
the symbols $||\cdot||_{\cO}$ and $(\cdot,\cdot)_{\cO}$
denote the norm and the inner product in $L_2(\cO)$.
The subscripts in $(\cdot,\cdot)_{\cO}$ and  $||\cdot||_{\cO}$
will be often omitted when apparent from the context.
We  denote $||\cdot||_{\si, \cO}$ the norm in the $L_2$-based
Sobolev space $H^\si(\cO)$.
Here we will have either $\cO=\Om$ or $\cO=\Gamma_0$.
We also denote by $H_0^\si(\cO)$ the closure of $C^\infty_0(\Om)$ in
$H^\si(\cO)$.

\subsection{Basic assumption}
We shall impose the following basic assumptions on the nonlinear
functions $g$ and $f$ which influence the wave component of the system.
\begin{assumption}
\label{hypo-0}
\begin{itemize}
\item
$g\in C(\mathbb{R})$ is a  non-decreasing function,
$g(0) =0$,
and there exists a constant $C>0$ such that
\begin{equation}\label{lip-g1}
|g(s)|\le C\left(1+|s|^{p} \right),
\quad s\in\R,
\end{equation}
where $1\le p\le5$ when $n=3$, while $1\le p<\infty$ for $n=2$.
\item
$f\in Lip_{loc}(\mathbb{R})$, and there exists a positive constant $M$ such that
\be\label{h:growth-for-f}
|f(s_1)- f(s_2)| \le M\left(1+ |s_1|^{q}+|s_2|^{q}\right)|s_1-s_2|
\quad \textrm{for} \; s_1,s_2\in \R\,,
\ee
where $q\le 2$ when $n=3$, and $q<\infty$ for $n=2$.
Moreover, the following dissipativity condition holds true:
\be\label{h:dissipativity-for-f}
\mu := \frac12\liminf_{|s|\rarrow \infty} \frac{f(s)}{s} > 0\,.
\ee
\item $p_0\in L_2(\Gamma_0)$.
\end{itemize}
\end{assumption}
Notice that the growth of both the nonlinearity (`source') $f$ 
and the damping $g$ affecting the wave component are allowed to 
be {\em critical} (see, e.g., \cite{chu-las-1,mem} for a discussion 
on this issue).

\subsection{Abstract formulation}
To study the dynamics of the PDE problem \eqref{pde-system}--\eqref{pde-i-d},
it is useful to recast it as an abstract evolution in an appropriate Hilbert
space.
The operators and spaces needed for this abstract set-up
are the following.
Let $A: \cD(A)\subset L_2(\Omega)\to L_2(\Omega)$
be the positive self-adjoint operator defined by
\[
A h = - \Delta h+\mu h\,, \quad \cD(A) =
\Big\{
h \in H^2(\Omega): \,\frac{\partial h}{\partial\nu}\Big|_{\Gamma} = 0\Big\}\,;
\]
where $\mu>0$ is given by (\ref{h:dissipativity-for-f}).
Next, let $N_0$ be the Neumann map from $L_2(\Gamma_0)$ to $L_2(\Omega)$,
defined by
\[
\psi = N_0 \phi \Longleftrightarrow
\Big\{(-\Delta +\mu) \psi = 0 \mbox{ in } \Omega\,;\
\frac{\partial\psi}{\partial\nu} \Big|_{\Gamma_0} = \phi\,, \
\frac{\partial\psi}{\partial\nu}\Big|_{\Gamma_1}
=0\Big\}\,.
\]
It will be used that
\be\label{an-property}
A^{3/4-{\epsilon}} N_0 \mbox{ continuous}:
L^2(\Gamma_0) \rarrow L^2(\Omega)\,,
\ee
which readily follows from the well known (see, e.g., \cite{las-trig-books}) 
regularity property
\[
N_0\;\mbox{continuous }: L_2(\Gamma_0) \rarrow
H^{3/2} (\Omega)\subset \cD(A^{3/4-{\epsilon}}),
\quad {\epsilon} > 0\,.
\]
It is also worth recalling the trace result
\be\label{traces-properties}
N_0^* A h = h|_{\Gamma_{0}} \quad \mbox{for}\quad h\in \cD(A)\,,
\ee
which can be extended by continuity to all $h \in H^{1}(\Omega)$.

Regarding the plate model, let us introduce
$\cA: \cD(\cA)\subset L_2(\Gamma_0) \rarrow L_2(\Gamma_0)$
as the positive, self-adjoint operator defined by
\[
\cA w = -\Delta w\,, \quad \cD(\cA) = H^2(\Gamma_0)\cap H^1_0(\Gamma_0)\,.
\]
It is well known that the fractional powers of $\cA$ are well defined;
we have, in particular,
\[
\|\cA^{1/2} v\|_{\Gamma_0} = ||\nabla v||_{\Gamma_0}
\quad \textrm{for any } v\in \cD(\cA^{1/2}) = H^1_0(\Gamma_0)\,.
\]
Finally, let us define the inertia operator $M_{\gamma} =I+\gamma \cA$, with obvious domain.
Then, according to the values of $\gamma$, one has
\begin{equation} \label{dom-mhalf}
V_\gamma :=\cD(M_{\gamma}^{1/2})\equiv
\begin{cases}\cD(\cA^{1/2})= H^1_0(\Gamma_0)&\gamma>0\,,
\\
L_2(\Gamma_0)& \gamma=0\,;
\end{cases}
\end{equation}
later we shall also need the dual space $V'_\gamma$
(where duality is with respect to the pivot space $L_2(\Gamma_0)$,
and we have $V_\gamma\subseteq L_2(\Gamma_0)\subseteq V'_\gamma$).

With the above dynamic operators, the initial/boundary value problem
\eqref{pde-system}--\eqref{pde-i-d} can be rewritten as the following
abstract second order system:
\begin{subequations}\label{abstract-system}
\bey
& & z_{tt} + A \left( z - \alpha\kappa  N_0 v_t\right) +  D(z_t) + F_1(z)=0\,,
\label{wave-eq} \\[2mm]
& & M_\gamma v_{tt} + \cA^2 v + \beta \kappa N_0^* A z_t
-\cA \theta + F_2(v)=0\,,
\label{plate-eq} \\[2mm]
& & \theta_{t} + \cA\theta + \cA v_t=0\,,
\label{heat-eq} \\[2mm]
& & z(0) = z^0\, \, z_t(0) = z^1\,; \;
v(0) = v^0\,,\, v_t(0) = v^1\,, \theta(0)=\theta^0\,,
\label{initial-data}
\eey
\end{subequations}
where we have introduced the Nemytski operators
\begin{equation}\label{Dh-mem}
D(h) := g(h)\,, \quad F_1(z) =  f(z)-\mu z\,,
\end{equation}
in \eqref{wave-eq}, whereas
\[
F_2(v)= -\left(Q-||\cA^{1/2}v||_{\Gamma_0}^2\right)\,\cA v - p_0
\]
in \eqref{plate-eq}.
Regarding the nonlinear force terms we have that
\be\label{wave-potential}
F_1(z) =  \Pi'(z) \quad\textrm{with}\quad
\Pi(z)=\int_\Om\int_0^{z(x)} \left( f(\xi)-\mu \xi\right) d\xi\; dx\,,
\ee
where $'$ stands for the Fr\'echet derivative (in an appropriate space).
It readily follows from \eqref{h:dissipativity-for-f} that
\be\label{estimate:wave-potential}
\Pi(z)\ge \delta_f \|z\|^2_{\Om} -M_f,\quad z\in H^1(\Om),
\ee
for some positive constants $\delta_f$ and $M_f$.
Similarly, we have that
\be\label{plate-potential}
F_2(v) =  \Phi'(v)\quad\textrm{with}\quad
\Phi(v) = \frac14 ||\cA^{1/2}v||_{\Gamma_0}^4 -\frac{Q}2
||\cA^{1/2}v||_{\Gamma_0}^2 - (p_0,v)\,.
\ee

The state spaces $Y_1$ for the wave component $[z,z_t]$ and
$Y_2$ for the plate component $[v,v_t]$ of system \eqref{abstract-system}
are given by
\[
\ba{l}
Y_1 := \cD(A^{1/2}) \times L_2(\Omega)\equiv H^1(\Om)\times L_2(\Omega)\,,
\\[1mm]
Y_2 := \cD(\cA) \times V_{\gamma}
\equiv \left[H^2(\Gamma_0)\cap H^1_0(\Gamma_0)\right]
\times \cD(M_{\gamma}^{1/2})\,,
\end{array}
\]
with respective norms
\[
||(z_1,z_2)||_{Y_1}^2=\|A^{1/2}z_1\|_{\Om}^2 +\| z_2\|_{\Om}^2\,,
\quad
||(v_1,v_2)||_{Y_2}^2=\|\cA v_1\|_{\Gamma_0}^2
+\|M_\gamma^{1/2} v_2\|_{\Gamma_0}^2\,,
\]
and $V_\gamma$ as in \eqref{dom-mhalf}.
The natural state space for the thermal component $\theta$
is $Y_3 =L_2(\Gamma_0)$.
The phase space for problem \eqref{abstract-system} is then
\be\label{state-space}
Y = Y_1 \times Y_2 \times Y_3= \cD(A^{1/2}) \times L^2(\Omega)
\times \cD(\cA ) \times V_{\gamma}\times L_2(\Gamma_0)\,,
\ee
endowed with the norm
\be\label{norm}
||y||_Y^2=||(z_1,z_2,v_1,v_2, \theta)||_Y^2
:= \beta ||(z_1,z_2)||_{Y_1}^2 + \alpha\left( ||(v_1,v_2)||_{Y_2}^2
+\|\theta\|^2_{\Gamma_0}\right)
\ee
(and obvious corresponding inner product).

An important consequence of Assumption \ref{hypo-0}
is that the nonlinear operator $F_1$ is locally Lipschitz
from $H^1(\Om)$ into $L^2(\Omega)$. Namely,
\be \label{F}
||F_1(z_1) - F_1(z_2)||_{\Omega}\le C(\rho) ||z_1-z_2||_{1,\Omega}\,,
\quad ||z_i||_{1,\Omega} \le \rho <\infty\,,
\ee
where $C(\rho)$ denotes a function that is bounded for bounded arguments.
It is important to emphasize that $F_1$ is bounded as an operator from
$H^1(\Om)$ into $L^2(\Omega)$, yet it is {\it not compact}.
This fact accounts for the adjective `critical' pertaining to the parameter
$q$ and the nonlinear term $F_1$.
It is worth noting that the mapping $F_2$ is critical in the case
$\gamma=0$ only.

The natural (nonlinear) energy functions associated with the solutions to the
{\em uncoupled} wave and plate models are given, respectively, by
\begin{subequations}\label{total-energies}
\bey
& & \cE_z(z(t),z_t(t)) := E_z^0(z(t),z_t(t)) + \Pi(z(t))\,,
\label{wave-tot-energy}
\\[2mm]
& & \cE_v(v(t),v_t(t)) := E_v^0(v(t),v_t(t)) + \Phi(v(t))\,,
\label{plate-tot-energy}
\eey
\end{subequations}
where we have set
\begin{subequations}\label{zero-energies}
\bey
& & E_z^0(t)\equiv E_z^0(z(t),z_t(t)) = \frac12\Big\{\|A^{1/2}z(t)\|_{\Omega}^2
+ \|z_t(t)\|_{\Omega}^2\Big\}\,,
\label{z-zero-energy}\\[2mm]
& & E_v^0(t)\equiv E_v^0(v(t),v_t(t)) = \frac12\Big\{\|\cA v(t)\|_{\Gamma_0}^2
+ \|M^{1/2}_\gamma v_t(t)\|_{\Gamma_0}^2\Big\}\,.
\eey
\end{subequations}
Since both energy functionals in \eqref{total-energies} may be negative,
it is convenient to introduce the following positive energy functions
\begin{eqnarray*}
& & E_z(z,z_t):=E_z^0(z,z_t)+\Pi(z)+M_f= \cE_z(z,z_t) +M_f\,,
\\[1mm]
& & E_v(v,v_t):=E_v^0(v,v_t)+\frac14 ||\cA^{1/2}v||^4= \cE_v(v,v_t)+
\frac{Q}2 ||\cA^{1/2}v||^2 + (p_0,v)\,,
\end{eqnarray*}
where $M_f$ is the constant in \eqref{estimate:wave-potential}.
The thermal energy is described by
$\cE_\theta(t):=\cE_\theta(\theta(t))  = \frac12\|\theta(t)\|^2_{\Gamma_0}$.
\\
Thus, let us introduce the total energy
$\cE(t)= \cE(z(t),z_t(t),v(t),v_t(t),\theta(t))$ of the system, namely
\be\label{total-energy}
\cE(t)=\cE(z(t),z_t(t),v(t),v_t(t),\theta(t))
:=\beta \cE_z(z,z_t) + \alpha
\Big( \cE_v(v,v_t)+\hf \|\theta\|^2_{\Gamma_0}\Big) \,,
\ee
whose positive part is given by
\be\label{total-energy-plus}
E(t)= E(z,z_t,v,v_t,\theta)
:=\beta E_z(z,z_t) + \alpha \Big(E_v(v,v_t)+\hf \|\theta\|^2_{\Gamma_0}\Big)\,.
\ee

It is easy to see from the structure of the energy functionals
and in view of (\ref{estimate:wave-potential}) and (\ref{plate-potential})
that for any $\al,\beta > 0$ there exist positive constants
$c$, $C$, and $M_0$ such that
\be\label{energy-bounds}
c E(z,z_t,v,v_t,\theta)-M_0\le \cE(z,z_t,v,v_t,\theta)
\le C E(z,z_t,v,v_t,\theta) + M_0\,,
\ee
where  $\cE$ and $E$ are the energies defined in \eqref{total-energy} and
\eqref{total-energy-plus}. One can also see that the constant $M_0$
depends linearly on $\alpha$ and $\beta$, i.e.
$M_0= \alpha M_0^1 +\beta M_0^2$.


\subsection{Well-posedness}
To study well-posedness of problem \eqref{pde-system}--\eqref{pde-i-d},
we may view the corresponding abstract system \eqref{abstract-system}
as a special case of a general second-order (in time) equation studied
in \cite{las-cmbs}.
This monograph includes local and global existence (and uniqueness) results
pertaining to the corresponding (strong and generalized) solutions.
The reader is referred to \cite[Section~2.6]{las-cmbs}, focused on
a structural acoustic model, yet not including thermal effects;
for the present case see Remark~2.6.2 and the specific references quoted
therein.
However, the `prototype' abstract equation explored in \cite{las-cmbs} is
motivated by PDE models which display nonlinear terms {\em on the boundary},
which renders the analysis more challenging and the application of
the results more involved than is actually needed in the present case.
As we shall see, the key feature here is that the first order system
corresponding to \eqref{abstract-system} is a
Lipschitz perturbation of a m-monotone system with suitable a-priori bounds.
Hence, in order to establish local existence and uniqueness of
the corresponding solutions we choose to invoke the recent result
\cite[Theorem 7.2]{chu-eller-las-1}.
For the reader's convenience, the proof of well-posedness is outlined
in the Appendix.

\smallskip
In order to make our statements precise, we need to
introduce the concepts of strong and generalized solutions.
\begin{definition}\label{str-sol-2ord}
A triplet of functions $(z(t),v(t),\theta(t))$
which satisfy the initial conditions \eqref{initial-data} and such that
\[
(z(t),v(t))\in C([0,T],\cD (A^{1/2})\times \cD (\cA))\cap
 C^1([0,T],L_2(\Om)\times V_\gamma)
\]
and $\theta\in C([0,T],L_2(\Gamma_0))$ is said to be
\begin{enumerate}
\item[{\bf (S)}] a {\em strong solution} to
problem (\ref{abstract-system})  on the interval $[0,T]$, iff
\begin{itemize}
  \item for any $0<a<b<T$ one has
\[
(z_t(t),v_t(t))\in L_1([a,b],\cD (A^{1/2})\times \cD(\cA^{1/2})),
\quad  \theta_t\in L_1([a,b],L_2(\Gamma_0))
\]
and
\[
(z_{tt}(t),v_{tt}(t))\in  L_1([a,b],L_2(\Om)\times V_\gamma)\,;
\]
\item
$A [z(t) -\al\kappa N_0v_t(t)] + D (z_t(t)) \in  L_2(\Om)$,
$ \cA^2 v(t)\in  V_\gamma'$ and  $\theta(t)\in\cD(\cA)$ for almost all $t\in [0,T]$;
\item equations \eqref{wave-eq} \eqref{plate-eq} and \eqref{heat-eq} are
satisfied in $L_2(\Om)\times V_\gamma'\times L_2(\Gamma_0)$
for almost all $t\in [0,T]$;
\end{itemize}
  \item[{\bf (G)}]
a {\em generalized solution} to problem~(\ref{abstract-system}) on the interval
$[0,T]$, iff there exists a sequence $\{(z_n(t),v_n(t)),\theta_n(t)\}_n$
of strong solutions to \eqref{abstract-system}, with initial data
$(z_n^0,z_n^1,v_n^0,v_n^1,\theta_n^0)$
(in place of $(z^0,z^1,v^0,v^1,\theta^0)$), such that
\[
\lim_{n\to\infty}
\max_{t\in[0,T]}\left\{ \|\partial_t z(t)-\partial_t z_n(t)\|_{\Om}
+ \| A^{1/2}\left(z(t)-z_n(t)\right)\|_{\Om}\right\}= 0,
\]
\[
\lim_{n\to\infty}
\max_{t\in[0,T]}\left\{ \|M_\gamma^{1/2}
\left(\partial_t v(t)-\partial_t v_n(t)\right)\|_{\Gamma_0}
+ \|\cA \left(v(t)-v_n(t)\right)\|_{\Gamma_0}\right\}= 0.
\]
and
\[
\lim_{n\to\infty}
\max_{t\in[0,T]}\| \theta(t)-\theta_n(t)\|_{\Gamma_0}= 0.
\]
\end{enumerate}
\end{definition}

In the statement of well-posedness of problem \eqref{pde-system}, we shall
also need the function space defined by
\begin{equation} \label{w-gamma}
W_\gamma:= \big\{\,u\in \cD(\cA): \; \cA^2u\in V'_\gamma\, \big\}\,,
\end{equation}
where $V'_\gamma$ denotes the dual space of $V_\gamma$ in \eqref{dom-mhalf}.
It is readily verified that
\begin{equation}\label{w-gamma1}
W_\gamma=
\begin{cases}
\cD(\cA^{3/2})\equiv
\left\{ u\in H^3(\Gamma_0): \; u =\Delta u=0
\; \textrm{on } \, \partial\Gamma_0\right\}
& \textrm{if $\gamma>0$}\,;
\\[1mm]
\cD(\cA^2)\equiv \left\{ u\in H^4(\Gamma_0): \; u =\Delta u=0
\; \textrm{on } \, \partial\Gamma_0\right\}
& \textrm{if $\gamma=0$}\,.
\end{cases}
\end{equation}

\begin{theorem}   \label{thm:well-posed}
Under Assumption~\ref{hypo-0} the PDE system~\eqref{pde-system} is
well-posed on
\[
Y=H^1(\Omega)\times L_2(\Omega)\times
[H^2(\Gamma_0)\cap H^1_0(\Gamma_0)]\times V_\gamma\times L_2(\Gamma_0),
\]
i.e. for any $(z^0,z^1,v^0,v^1,\theta^0)=:y_0\in Y$
there exists a unique generalized solution
$y(t)= (z(t),z_t(t),v(t),v_t(t),\theta(t))$ which depends
continuously on initial data.
This solution satisfies the energy \emph{inequality}
\be\label{energy-ineq}
\cE(t) + \beta \, \int_s^t (D(z_t),z_t)_\Om\,d\tau
+ \alpha \, \int_s^t \|\cA^{1/2} \theta\|^2_{\Gamma_0}\,d\tau \le \cE(s)\,,
\qquad 0\le s\le t\,,
\ee
with the total energy $\cE(t)$ given by \eqref{total-energy}.
Moreover, if initial data are such that
\[
z^0\,,\, z^1\in \cD(A^{1/2})\,,\quad
v^0\in W_\gamma\,, \; v^1\in \cD(\cA)\,,
\quad  \theta^0\in \cD(\cA)
\]
and
\[
A [z^0 -\al\kappa N_0v^1]+   D (z^1) \in  L_2(\Om)\,,
\]
then there exists a unique strong solution $y(t)$ satisfying the
{\em energy identity:}
\be\label{energy-identity}
\cE(t) + \beta \, \int_s^t (D(z_t),z_t)_\Om\,d\tau
+ \alpha \,  \int_s^t \|\cA^{1/2} \theta\|^2_{\Gamma_0}\,d\tau   = \cE(s)\,,
\qquad 0\le s\le t\,.
\ee
Both strong and generalized solutions satisfy the inequality
\begin{equation}\label{energy-decr}
\cE(t)\le \cE(s)
\quad \textrm{for }\; t\ge s\,,
\end{equation}
where $\cE(t)\equiv \cE(z(t),z_t(t),v(t),v_t(t),\theta(t))$.
This implies, in particular,
\begin{equation}\label{energy-est2}
E(z(t),z_t(t),v(t),v_t(t),\theta(t))
\le C\left( 1+ E(z^0,z^1,v^0,v^1, \theta^0)\right)\,,
\quad \textrm{for }\; t\ge 0
\end{equation}
(see \eqref{energy-bounds}).
\end{theorem}

We shall sketch in the Appendix a self-contained proof, tailored for
the specific problem under investigation.

\begin{remark}\label{re:weak-so}
\begin{rm}
The existence of generalized solutions established in
Theorem~\ref{thm:well-posed} is obtained by using the
theory of nonlinear semigroups.
These solutions are defined as strong limits of regular (strong)
solutions, as in the part $\mathbf{(G)}$ of Definition~\ref{str-sol-2ord}.
This does not necessarily mean that generalized solutions
satisfy a variational equality.
However, in view of the regularity of $g$ and $f$ we may compute appropriate
limits and obtain the variational form stated below.
Indeed, using similar arguments as in \cite{chu-las-jde06} one can prove that
any generalized solution $y(t)=(z(t),z_t(t),v(t),v_t(t),\theta(t))$
to problem \eqref{abstract-system} is also {\em weak}, i.e. it satisfies
the following system of (variational) equations:
\begin{subequations}\label{var-equations}
\bey
& &  \frac{d}{dt} (z_{t},\phi)_\Om + (\g z,\g\phi)_\Om +  (g(z_t),\phi)_\Om
\nonumber
\\
& & \myspace
 - \alpha\kappa (v_t,\phi)_{\Gamma_0} + (f(z),\phi)_\Om=0\,,
\label{wave-eq-var}
\\[1mm]
& &  \frac{d}{dt}(\Mg v_{t}+\beta \kappa  z,\psi)_{\Gamma_0}
+(\Delta v, \Delta \psi)_{\Gamma_0}-(\theta,\Delta  \psi)_{\Gamma_0}
\nonumber
\\
& & \myspace  +\left[Q -\|\nabla v\|_{\Gamma_0}^2 \right]
(\Delta v,\psi)_{\Gamma_0} = (p_0,\psi)_{\Gamma_0}\,,
\label{plate-eq-var}
\\[2mm]
& & \frac{d}{dt}(\theta,\chi)_{\Gamma_0}
+(\nabla\theta, \nabla\chi)_{\Gamma_0} +(\nabla v_t, \nabla\chi)_{\Gamma_0}=0
\label{heat-eq-var}
\eey
\end{subequations}
for any $\phi\in  H^1(\Om)$, $\psi\in  H^2(\Gamma_0)\cap H_0^1(\Gamma_0)$
and $\chi\in H_0^1(\Gamma_0)$
in the sense of distributions.
\end{rm}
\end{remark}

\smallskip
Theorem~\ref{thm:well-posed} enables us to define a dynamical system
$(Y,S_t)$ with the phase space $Y$ given by \eqref{state-space} and with
the evolution operator $S_t\, :\, Y\to Y$ given by the relation
\[
S_t y_0=(z(t),z_t(t),v(t),v_t(t),\theta(t))\,,
\quad y_0=(z^0,z^1,v^0,v^1,\theta^0)\,,
\]
where $(z(t),v(t),\theta(t))$ is a generalized solution to
\eqref{abstract-system}.
Moreover, the monotonicity of the damping operator $D$, combined with
the Lipschitz condition in \eqref{F} and the energy's bound \eqref{energy-est2}
imply, by a pretty routine argument,
that the semi-flow $S_t$ is locally Lipschitz on $Y$.
More precisely, there exists $a>0$ and $b(\rho)>0$ such that
\begin{equation}\label{SF}
||S_t y_1 - S_t y_2 ||_Y \leq a\, e^{b(\rho) t}
|| y_1 - y_2 ||_Y,\quad \forall\; \|y_i\|_Y\le \rho,\; t\ge 0\,.
\end{equation}

\subsection{Energy functionals and stationary solutions}
We conclude this section by discussing several properties of the energy
functionals and stationary solutions.
\\
It follows from \eqref{energy-bounds} that the energy
$\cE(z_0,z_1,v_0,v_1,\theta_0)$ is bounded from below on $Y$
and $\cE(z_0, z_1, v_0, v_1,\theta_0)\to +\infty$ when
$\|(z_0, z_1, v_0, v_1,\theta_0)\|_Y\to +\infty$.
This, in turn, implies that there exists $R_*>0$ such that the set
\begin{equation}\label{wr-set}
\cW_R=\left\{ y=(z_0, z_1, v_0, v_1,\theta_0)\in Y\, :\;
\cE(z_0, z_1, v_0, v_1,\theta_0)\le R \right\}
\end{equation}
is a non-empty bounded set in $Y$ for all $R\ge R_*$. Moreover
any bounded set $B\subset Y$ is contained in $\cW_R$ for some $R$ and,
as it follows from  (\ref{energy-decr}),
the set $\cW_R$
is forward invariant \wrt~the semi-flow $S_t$, i.e. $S_t\cW_R\subset\cW_R$ for all
$t>0$.
Thus, we can consider the restriction $(\cW_R, S_t)$ of the dynamical system
$(Y, S_t)$ on $\cW_R$, $R\ge R_*$.

\medskip
We introduce next the set of stationary points of $S_t$ denoted by $\cN$,
\[
 \cN=\left\{ V\in Y\; :\; S_tV=V\mbox{~for all~} t\ge 0\right\}\,.
\]
Every stationary point $V$ has the form $V=(z,0,v,0,0)$,
where $z\in H^1(\Om)$ and  $v\in  H^2(\Gamma_0)\cap H_0^1(\Gamma_0)$ are,
respectively, weak (variational) solution to the problems
\[
- \Delta z + f(z) = 0 \quad \textrm{in} \;\Omega\,,
\qquad \frac{\partial z}{\partial\nu}  = 0
\quad \textrm{on} \;\Gamma\,,
\]
and
\[
\left\{
\ba{lll}
\displaystyle
\Delta^2 v + \Big[Q -\int_{\Gamma_0}|\nabla v(x,t)|^2 dx\Big]
\,\Delta v   = p_0 & & {\rm on}\; \Gamma_0
\\[3mm]
v=\Delta v = 0 & & {\rm on}\; \partial\Gamma_0\,.
\ea
\right.
\]
It is clear that the set $\cN$ of stationary points does not depend on
$\gamma$ and $\kappa$. Therefore
using the properties of the potentials $\Pi$ and $\Phi$ given by
\eqref{wave-potential} and \eqref{plate-potential} one can easily
prove the following assertion.

\begin{lemma} \label{lemma:b-stat}
Under Assumption~\ref{hypo-0} the set $\cN$ of stationary points
for the semi-flow $S_t$ generated by equations (\ref{abstract-system})
is a closed bounded set in $Y$, and hence there exists $R_{**}\ge R_*$ (independent
of $\gamma$ and $\kappa$)
such that $\cN\subset\cW_R$ for every $R\ge R_{**}$.
\end{lemma}

Later we will also need the notion of {\it unstable manifold}
$M^u(\cN)$ emanating from the set $\cN$, which is defined as the set of
all $W\in Y$ such that there exists a full trajectory
$\gamma=\{ W(t)\, :\, t\in\R\}$ with the properties
\[
W(0)=W\mbox{~and~} \lim_{t\to -\infty}{\rm dist}_{Y}(W(t),\cN)=0\,.
\]
We finally recall that a continuous curve
$\gamma=\{ W(t)\, :\, t\in\R\}\subset Y$
is said to be a full trajectory if $S_tW(\tau)=W(t+\tau)$ for any
$t\ge 0$ and $\tau\in\R$.

\section{The statement of main results}\label{s:statement}
The goal of the present paper is to show the existence of a global attractor
for the dynamical system generated by problem \eqref{pde-system}, and to
study its properties.

Let us recall ({\em cf}. \cite{babin,chueshov-book,hale,Tem97})
that a global attractor for a dynamical system $(X,S_t)$
on a complete metric space $X$ is a closed bounded set $\mfA$ in
$X$ which is invariant (i.e. $S_t\mfA =\mfA$ for any $t>0$) and
uniformly attracting, i.e.
\[
\lim_{t\to +\infty}\sup_{y\in B} {\rm dist}_X
\{ S_ty, \mfA\}=0\quad \mbox{ for any bounded set}\quad B\subset X\,.
\]
The {\em fractal dimension} $\dim_f M$ of a compact set $M$ is defined by
\[
\dim_fM=\limsup_{\eps\to 0}\frac{\ln N(M,\eps)}{\ln (1/\eps)}\,,
\]
where $N(M,\eps)$ is the minimal number of closed sets of
diameter $2\eps$ which cover the set $M$.

\smallskip
To prove the existence of a global attractor for problem
\eqref{abstract-system} we need an additional hypothesis concerning
the damping function $g$.

\begin{assumption}  \label{hypo-1}
Besides to Assumption~\ref{hypo-0}, suppose that
for any $\eps>0$ there exists $c_\eps>0$ such that
\begin{equation}\label{lip-g2}
s^2 \le \eps + c_\eps  s g(s)\quad \textrm{for} \quad  s\in\R\,.
\end{equation}
\end{assumption}
Let us recall from \cite[Remark~3.2]{bcl-cpaa-06} that \eqref{lip-g2} holds true,
e.g., when
(i) $g(s)$  is non-decreasing on $\R$ and
{\em strictly} increasing in some (small) neighbourhood of $0$, and
(ii) we have that
\[
\liminf_{|s|\to\infty}\frac{g(s)}{s}>0\,.
\]
In particular, Assumption~\ref{hypo-1} allows the damping function $g$
to be constant on some closed finite intervals which are away from zero.

Our first main result provides the existence of a global attractor for problem
\eqref{abstract-system}, as well as a description of its structure.

\begin{theorem}\label{t:main}
Under Assumption~\ref{hypo-1}  the dynamical
system $(Y,S_t)$ generated by problem \eqref{abstract-system} has a
compact global attractor $\mathfrak{A}$ which coincides with the unstable
manifold $M^u(\cN)$ emanating from the set $\cN$ of stationary points for
$S_t$, namely $\mathfrak{A}\equiv M^u(\cN)$.
Moreover,
\begin{equation}\label{stb-to-n}
\lim\limits_{t\to +\infty}{\rm dist}_{Y}(S_tW,\cN)=0\quad
\mbox{for any}\quad W \in Y\,,
\end{equation}
and for any trajectory $U= (z(t),z_t(t),v(t),v_t(t),\theta(t))$
from the attractor $\mfA$ we have that
\begin{equation} \label{first-bnd}
\|z\|_{1,\Om}^2 +\|z_t\|_\Om^2
+ \|\Delta v\|_{\Gamma_0}^2+ \|v_{t}\|_{\Gamma_0}^2+\gamma\|\nabla v_{t}\|_{\Gamma_0}^2
+\|\theta\|_{\Gamma_0}^2\le \widetilde{R}^2\,,
\end{equation}
where $\widetilde{R}$ does not depend on $\gamma,\kappa\in [0,1]$.
\end{theorem}

\smallskip
Our second main result contains specific assertions regarding the
dimension and regularity of the attractor $\mathfrak{A}$.
These require the following additional assumptions.

\begin{assumption}
\label{hypo-2}
Besides to Assumption~\ref{hypo-1}, let the following conditions hold:
\begin{itemize}
  \item there exist positive constants $m$ and $M$ such that
\begin{equation}\label{g-for-dim0}
m\le \frac{ g(s_1)-g(s_2)}{s_1-s_2}\le M\left(
1+ s_1g(s_1)+ s_2g(s_2)\right)^\sigma,
\quad s_1,s_2\in\R,\;   s_1\neq s_2\,,
\end{equation}
where $0\le \sigma<1$ is  arbitrary in the case $n=2$ and $\sigma=2/3$
in the case $n=3$;
  \item $f\in C^2(\R)$ and $|f''(s)|\le C(1+|s|^{q-1})$ for $s\in\R$, where
  $q=2$ for $n=3$ and $1\le q<\infty$ for $n=2$.
\end{itemize}
\end{assumption}

\smallskip
As already observed in \cite[Remark~3.5]{bcl-cpaa-06}, we note that
if $g\in C^1(\R)$, then \eqref{g-for-dim0} is equivalent to the requirement
\begin{equation}\label{g-for-dim1}
m'\le g'(s)\le M'\left( 1+ sg(s)\right)^\sigma
~~\mbox{for all}~~s\in\R\,,
\end{equation}
for some constants $m'$, $M'>0$. Moreover, one can see that the inequality
on the right hand side of \eqref{g-for-dim1} holds true if we assume that
\begin{equation}\label{g-for-dim2}
m'<g'(s)\le M'\left( 1+ |s|^{p-1}\right)  ~~\mbox{and}~~
s g(s)\ge m'' |s|^{(p-1)/\sigma } -M'' ~~\mbox{for all}~~
s\in\R
\end{equation}
for some $m''>0$.
Moreover, the second requirement in \eqref{g-for-dim2}
follows from the first one if $1\le p\le 1+2\sigma$.

\smallskip

\begin{theorem}\label{t:main2}
Let Assumption~\ref{hypo-2} hold. Then the compact global attractor
$\mathfrak{A}$ given by Theorem~\ref{t:main}
possesses the following properties:
\begin{itemize}
    \item[{\bf 1.}] The attractor $\mfA$ has a finite fractal dimension,
and there exists a constant $d$ independent of $\gamma$ and $\kappa$, such that
$dim_f\mfA\le d$.
    \item[{\bf 2.}] The attractor $\mfA$ is a bounded set in the space
\[
Y_*=W^{2}_{6/p}(\Om)\times \cD(A^{1/2}) \times W_\gamma \times \cD(\cA)
\times  \cD(\cA) \
\]
in the case $n=3$ and $3<p\le 5$,
and in the space
\[
Y_{**}=H^2(\Om)\times \cD(A^{1/2})\times W_\gamma \times \cD(\cA)
\times  \cD(\cA)
\]
in the other cases, where $ W_\gamma$ is given by \eqref{w-gamma1}.
Moreover, for any trajectory
\[
U= (z(t),z_t(t),v(t),v_t(t),\theta(t)), \quad t\in\R,
\]
from the attractor $\mfA$ we have that
\begin{eqnarray} \label{regularity-estimate}
\|z(t)\|_{W^2_{p_*}(\Om)}^2+\|z_t(t)\|_1^2+\|z_{tt}(t)\|^2  +
\|v(t)\|_3^2+\|\Delta v_t(t)\|^2 & & \\ +\|v_{tt}(t)\|^2+ \gamma\|\nabla v_{tt}(t)\|^2
+\|\theta_t(t)\|^2 +\|\theta(t)\|_2^2 &\le & R_*^2\,, \nonumber
\end{eqnarray}
where  $p_*= \min\{2,6/p\}$ and $R_*$ does not depend on $\gamma,\kappa\in [0,1]$ (in the case $\gamma=0$
we have the additional estimate $\|v\|_4\le R_*$).
\end{itemize}
\end{theorem}

\smallskip
In the case $\kappa=0$ Theorem~\ref{t:main} and Theorem~\ref{t:main2}
provide specific assertions for either uncoupled equation.
We record explicitly for the reader's convenience the one pertaining to the
wave component, namely Corollary~\ref{c:wave} below, obtained earlier in
\cite{bcl-cpaa-06}.
The reader is referred to
\cite{bcl-cpaa-06} for a comparison between the statements in
Corollary~\ref{c:wave} and the previous literature on the long-time
behaviour of semilinear wave equations with interior nonlinear dissipation.

\begin{corollary}[\cite{bcl-cpaa-06}]  \label{c:wave}
Let $f$ and $g$ satisfy the conditions in
Assump\-tion~\ref{hypo-1}.
Then the dynamical system $(Y_1,S^1_t)$ generated by problem
\be\label{wave-system}
\left\{
\ba{lll}
z_{tt}+  g(z_t)  -  \Delta z + f(z) = 0
& {\rm in}\; \Omega\times (0,T)
\\[2mm]
\displaystyle\frac{\partial z}{\partial\nu}  = 0
& {\rm on}\; \Gamma\times (0,T)
\ea
\right.
\ee
possesses a compact global attractor $\mathfrak{A}_1\equiv M^u(\cN_1)$,
where $\cN_1$ is the set of equilibria for \eqref{wave-system}.
If $f$ and $g$ satisfy  Assumption~\ref{hypo-2},
then (i) the attractor $\mfA_1$ has a finite fractal dimension;
 and (ii)   $\mfA_1$ is a bounded set in
          the space $W^2_{6/p}(\Om)\times \cD(A^{1/2})$
in the case $n=3$ and $3<p\le 5$, and in the space
$\cD(A)\times \cD(A^{1/2})$
in other cases.
\end{corollary}

\smallskip
As for the thermoelastic Berger plate model, Theorems~\ref{t:main}
and~\ref{t:main2} yield the following result, which is---to the authors'
best knowledge---completely new (a similar result for thermoelastic von Karman
equations was established recently in \cite{chu-las-preprint06}).

\begin{corollary}\label{c:plate}
Suppose that $p_0\in L_2(\Gamma_0)$.
Then  for every $\gamma\ge 0$ the dynamical system $(Y_2\times L_2(\Gamma_0),S^2_t)$ generated by
problem
\be\label{plate-system}
\left\{
\ba{lll}
\displaystyle
v_{tt}-\gamma\Delta v_{tt} + \Delta^2 v
+ \Big[Q -\int_{\Gamma_0}|\nabla v(x,t)|^2 dx\Big] \,\Delta v+\Delta\theta  = p_0
& & {\rm on}\; \Gamma_0\times(0,T) \\
\displaystyle
\theta_{t} - \Delta\theta- \,\Delta v_t = 0
& & {\rm on}\; \Gamma_0\times(0,T)
\\[2mm]
v=\Delta v = 0,\; \theta =0 & & \!\!\!{\rm on}\; \partial\Gamma_0\times(0,T)
\ea
\right.
\ee
has a compact global attractor $\mathfrak{A}_2\equiv M^u(\cN_2)$,
where $\cN_2$ is the set of equilibria for \eqref{plate-system}.
In addition, the attractor $\mfA_2$ is a bounded
set in $W_\gamma\times \cD(\cA) \times \cD(\cA)$ and
has a finite fractal dimension.
\end{corollary}

The following Theorem provides us with more detailed information regarding
the dependence of the attractor on the parameters $\gamma$ and $\kappa$.

\begin{theorem}\label{co:up-sc}
Let Assumption~\ref{hypo-2} be in force and let $p<5$.
Denote by $S_t^{\gamma,\kappa}$ the evolution operator of
problem \eqref{abstract-system} in the space
\[
Y_\gamma := Y= \cD(A^{1/2}) \times L^2(\Omega)
\times \cD(\cA ) \times V_{\gamma}\times L_2(\Gamma_0)\,.
\]
Let $\mfA^{\gamma,\kappa}$ be a global attractor for the system
$(Y_\gamma, S_t^{\gamma,\kappa})$. Then
the family of the attractors $\mfA^{\gamma,\kappa}$
is upper semi-continuous on $\La:=[0,1]\times [0,1]$: namely,
for any $\la_0=(\gamma_0,\kappa_0)\in\La$ we have that
\begin{equation}\label{semi-cont-at1}
\lim_{(\gamma,\kappa)\to \la_0}
\sup\left\{ {\rm dist}_{Y_{\ga_0}}\left(U, \mfA^{\la_0} \right)\, :
U\in \mfA^{\gamma,\kappa}\right\}=0\,.
\end{equation}
In particular,
\begin{itemize}
  \item
if $\kappa_0=0$ one has
\begin{equation}\label{semi-cont-at2}
\sup\left\{ {\rm dist}_{Y_{\ga_0}} \left(U,\mfA_1\times \mfA^{\gamma_0}_2 \right)
:\,  U\in \mfA^{\gamma,\kappa}\right\}=0 \quad  as \quad
\gamma\to \gamma_0,\; \kappa\to 0\,,
\end{equation}
where $\mfA_1$ is the same as in Corollary~\ref{c:wave}
and
$\mfA^{\gamma}_2$ is the attractor generated by problem \eqref{plate-system};
  \item
if $\gamma_0=\kappa_0=0$ and $\kappa\equiv 0$, then
\begin{equation}\label{semi-cont-at2-0}
\sup\left\{ {\rm dist}_{\cH}\left(V,  \mfA^{0}_2 \right)\, :
V\in \mfA_2^{\gamma}\right\}=0 \quad  as \quad
\gamma\to 0\,,
\end{equation}
where $\cH=\cD(\cA)\times L_2(\Gamma_0)\times L_2(\Gamma_0)$.
\end{itemize}
\end{theorem}

We shall give a proof of this Theorem at the end of Section~\ref{s:proof-t2}.

\begin{remark}\label{re:up-sc1}
\begin{rm}
{\bf 1.} Assertion \eqref{semi-cont-at2} means that in the
`decoupling limit' $\kappa\to 0$ the global attractor $\mfA$ becomes
close to the cartesian product of global attractors pertaining to the
problems \eqref{wave-system} and \eqref{plate-system}, respectively.
Condition \eqref{semi-cont-at2-0} states the semi-continuity of the attractor
corresponding to the uncoupled system \eqref{plate-system},
as the parameter $\gamma$ (rotational inertia) tends to zero.
This result is---to the authors's best knowledge---completely new for the
thermoelastic Berger plate model.

{\bf 2.} It should be noted that in the critical case $p=5$ the semi-continuity
property \eqref{semi-cont-at1} may be established as well, though in the
({\em weaker}) topology of the space
\begin{equation}\label{Y-g-eps}
Y_\gamma^{-\eps} := Y= \cD(A^{1/2-\eps}) \times L^2(\Omega)
\times \cD(\cA ) \times V_{\gamma}\times L_2(\Gamma_0)\,,
\end{equation}
with arbitrary $\eps>0$; see Remark~\ref{re:up-sc2}.
Whether we can take $\eps=0$ in the case $p=5$ is still an open question.
\end{rm}
\end{remark}

\section{Main inequality}\label{s:main-ineq}
This section is entirely devoted to show a preliminary inequality,
which constitutes a fundamental common step for the proofs of both
Theorems~\ref{t:main} and~\ref{t:main2}.
This inequality, namely \eqref{ineq:main} in Proposition \ref{pr:main}
below, holds true under the basic Assumption~\ref{hypo-0}.
Yet, as we shall see in the next sections, by progressively strengthening
the assumptions on the nonlinearities it will yield the key estimates
specifically used for the proof of Theorems~\ref{t:main} and~\ref{t:main2}.

\begin{proposition}\label{pr:main}
Let Assumption~\ref{hypo-0} hold.
Assume that $y_1, y_2\in \cW_R$ for some $R>R_*$,
where $\cW_R$ is defined by (\ref{wr-set}) and denote
$$
(h(t),h_t(t),u(t),u_t(t),\psi(t)):=S_t y_1\,,\quad
(\zeta(t),\zeta_t(t),w(t),w_t(t),\xi(t)):=S_t y_2.
$$
Let
\begin{equation}\label{zvth-dif}
z(t):=h(t)-\zeta(t),\quad  v(t):=u(t)-w(t),\quad \theta(t):=\psi(t)-\xi(t)\,.
\end{equation}
There exist $T_0>0$ and positive constants $c_0$, $c_1$ and $c_2(R)$
independent of $T$ and $\gamma, \kappa\in [0,1]$ such that for every
$T\ge T_0$ the following inequality holds:
\begin{eqnarray}\label{ineq:main}
TE^0(T) +
\int_{0}^{T} E^0(t) dt
\le c_0 \Big[ \int_0^T\left( \|z_t\|^2+ \alpha\|\cA^{1/2}\theta\|^2\right)dt
+\beta G_0^T(z) \Big]
\nonumber \\
+  c_1
\left[ H_0^T(z)+ \Psi_T(z,v) \right]+
c_2(R) \int_0^T\left( \|z\|^2+ \|v\|^2\right)dt\,,
  \end{eqnarray}
where $E^0(t)=\beta E^0_z(t)+\alpha E^0_v(t)+(\alpha/2) \|\theta(t)\|^2$ with
$E^0_z(t)$ and $E^0_v(t)$ given by \eqref{zero-energies}.
We also introduced the notations
\be\label{integrals-st}
G_s^t(z)=\!\int_s^t\!\!\big(D(\zeta_t+z_t)-D(\zeta_t),z_t\big)_{\Omega} d\tau\,,
\ee
\be\label{integrals-st-z}
H_s^t(z)=\!\int_s^t\!\!\big|\big(D(\zeta_t+z_t)
-D(\zeta_t),z\big)_{\Omega}\big| d\tau\,,
\ee
and
\bey
\Psi_T(z,v) & = &
\beta\Big|\int_0^T (\cF_1(z),z_t)dt\Big|
+ \beta\Big|\int_0^T\!\!\int_t^T (\cF_1(z),z_t)\,d\tau \,dt\Big|
\nonumber\\[1mm]
& & \qquad + \alpha\Big|\int_0^T (\cF_2(v),v_t)dt\Big|
+ \alpha\Big|\int_0^T\!\!\int_t^T (\cF_2(v),v_t)\,d\tau \,dt\Big|\,,
\label{psi-def}
\eey
with
\be\label{cal-f}
\cF_1(z)= F_1(\zeta+z)-F_1(\zeta)
\quad\mbox{and}\quad
\cF_2(z)= F_2(\zeta+z)-F_2(\zeta)\,,
\ee
where $F_1$ and $F_2$ are the same as in \eqref{abstract-system}.
\end{proposition}

\proof
{\bf Step 1 (Energy identity).}
Without loss of generality, we can assume that $(h(t),u(t),\psi(t))$ and
$(\zeta(t),w(t),\xi(t))$ are strong solutions.
By the invariance of $\cW_R$ and in view of relation \eqref{energy-bounds}
there exists a constant $C_{R}>0$ such that
\be\label{energ-bnd}
E^0_h(h(t),h_t(t))+ E^0_{\zeta}(\zeta(t),\zeta_t(t))
+ E^0_u(u(t),u_t(t))+ E^0_w(w(t),w_t(t))
\le C_{R}
\ee
for all $t\ge 0$,  where $E_z^0(t)$ and $E_v^0(t)$  denote the corresponding
(free) energies as in \eqref{zero-energies}. We also have that
\be\label{energ-bnd2}
\|\psi(t)\|^2+ \|\xi(t)\|^2\le C_{R}\,,\qquad t\ge 0\,.
\ee
We establish first of an energy type equality regarding
\begin{equation}\label{e0-full}
E^0(t)=\beta E^0_z(t)+\alpha E^0_v(t)+\frac{\alpha}2 \|\theta(t)\|^2.
\end{equation}
\begin{lemma}\label{le:ener-z}
For any $T>0$ and all $t$, $0\le t\le T$, $E^0(t)$ satisfies
\begin{align}\label{key-equality}
E^0(T)+ &\beta G_t^T(z)+  \alpha  \int_t^T\|\cA^{1/2}\theta\|^2d\tau
\\
&= E^0(t)- \beta\int_t^T (\cF_1(z),z_t)\,d\tau
- \alpha \int_t^T (\cF_2(v),v_t)\,d\tau\,,
\nonumber
\end{align}
where $G_t^T(z)$ is  given by \eqref{integrals-st},
while $\cF_1(z)$ and $\cF_2(z)$ are defined by \eqref{cal-f}.
\end{lemma}

\proof
It is elementary to derive
the following system of coupled equations for the differences given in
\eqref{zvth-dif}:
\begin{subequations}\label{eqns-for-differences}
\bey
& & z_{tt} + A (z- \alpha\kappa  N_0 v_t) +  D(\zeta_t+z_t)-D(\zeta_t)  +
\cF_1(z)=0,
\label{z-eqn} \\[2mm]
& & M_{\gamma} v_{tt} + \cA^2 v  -\cA\theta + \beta\kappa N_0^* A z_t +
\cF_2(v)=0\,,
\label{v-eqn}
\\[2mm]
& & \theta_{t} + \cA \theta + \cA v_t =0\,,
\label{th-eqn}
\eey
\end{subequations}
with $\cF_i$ defined in \eqref{cal-f}.
Next, by standard energy methods we obtain
\begin{eqnarray*}
& & E^0_z(T) + G_t^T(z)
= E^0_z(t)+\alpha\kappa \int_t^T (AN_0v_t,z_t)\,d\tau
- \int_t^T (\cF_1(z),z_t)\,d\tau\,;
\\[1mm]
& & E^0_v(T) = E^0_v(t)-\beta\kappa \int_t^T (N_0^*A z_t,v_t)\,d\tau
-\int_t^T (\cF_2(v),v_t)\,d\tau +\int_t^T (\cA\theta,v_t)\,d\tau\,;
\\[1mm]
& & \hf \|\theta(T)\|^2+\int_t^T \|\cA^{1/2}\theta\|^2 \, d\tau = \hf \|\theta(t)\|^2 -
\int_t^T (\cA v_t, \theta)\,d\tau\,.
\end{eqnarray*}
Then, \eqref{key-equality} readily follows multiplying the first equation
by $\beta$, the second and third one by $\alpha$ and summing up.
\qed

\smallskip
{\bf Step 2. (Reconstruction of the energy integral) }
We return to the coupled system \eqref{eqns-for-differences} satisfied by
$(z,v,\theta)$.
We multiply equation \eqref{z-eqn}  by $z$, and integrate between
$0$ and $T$, thereby obtaining
\bey
\lefteqn{\hspace*{-15mm}
\int_0^T \|A^{1/2}z\|^2 dt \le  c_0\left( E^0_z(T)+ E^0_z(0)\right) +
\int_0^T \|z_t\|^2 dt
 }
\nonumber \\[1mm]
& & \qquad +\, H_0^T(z) + \alpha\kappa \int_0^T |(v_t,N_0^*A z)|\,dt
+ \int_0^T |(\cF_1(z),z)|\,dt\,,
\label{eq:1-step2}
\eey
where $H_0^T(z)$ is defined in \eqref{integrals-st-z}, and $c_0$ does
not depend on $\alpha$, $\beta$, $\gamma$.
It is clear from \eqref{F}  and   \eqref{energ-bnd} that
\[
 |(\cF_1(z),z)|\le \frac{C_R}{\beta}\|A^{1/2}z\|\, \|z\|\,.
\]
On the other hand, using \eqref{an-property} with $\epsilon=1/4-\delta$, $0<\delta<1/4$,
and interpolation we have that
\begin{align*}
    |(v_t,N_0^*Az)|
\le \|v_t\|_{\Gamma_0}\, ||N_0^*A^{1/2+\delta}|| \,\|A^{1/2-\delta}z\|_{\Om}
\le C  \|v_t\|_{\Gamma_0} \,\|A^{1/2-\delta}z\|_{\Om}
\\[2mm]
\le\eps  \|v_t\|^2_{\Gamma_0} +\eps_1 \|A^{1/2}z\|^2_{\Om}+
C_{\eps,\eps_1} \|z \|^2_{\Om}\,,
\end{align*}
for any $\eps,\eps_1>0$.
Then, by appropriately choosing and rescaling $\eps$ and $\eps_1$
we obtain from \eqref{eq:1-step2} that
\bey\label{potential-for-z}
\lefteqn{
\int_0^T ||A^{1/2}z||^2 dt
\le  C_0\left( E^0_z(T)+ E^0_z(0)\right)  }
\\[1mm]
& & \quad +\,   \eps\int_0^T  \|v_t\|_{\Gamma_0}^2 dt
 +2 \int_0^T ||z_t||^2  dt +
 C_1 H_0^T(z)+ C_2(R,\eps)\int_0^T ||z||^2 dt\,,\nonumber
\eey
for any $\eps>0$.

For the plate component we use the multiplier $\eps v+ \cA^{-1}\theta$.
(The {\em abstract} multiplier $\cA^{-1}\theta$ was introduced
in \cite{avalos-las-2}).
One gets, initially,
\begin{eqnarray*}
\lefteqn{
\eps(M_{\gamma} v_{tt},v)+ \eps\|\cA v\|^2 -\epsilon (\cA \theta, v)
+ (M_{\gamma} v_{tt},\cA^{-1}\theta) +(\cA v,\theta)}
\\[2mm]
& & \qquad
-\|\theta\|^2 + \beta \kappa (z_t,\eps v+ \cA^{-1}\theta)_{\Gamma_0}
+ (\cF_2(v),\epsilon v+\cA^{-1}\theta)= 0\,,
\end{eqnarray*}
which can be rewritten as
\begin{eqnarray}
\lefteqn{
\eps\frac{d}{dt}(M_{\gamma} v_t,v)-\eps \|M_{\gamma}^{1/2} v_t\|^2
+ \eps\|\cA v\|^2 +(1-\eps) (\cA \theta, v) }
\nonumber\\[2mm]
& & \quad
+ \,\frac{d}{dt}(M_{\gamma} v_t,\cA^{-1}\theta)+(M_{\gamma} v_t,\theta)
+ \|M_{\gamma}^{1/2}  v_t\|^2-\|\theta\|^2
\nonumber\\[2mm]
& & \qquad
+ \,\beta \kappa (z_t,\eps v+ \cA^{-1}\theta)_{\Gamma_0}
+ (\cF_2(v),\eps v+\cA^{-1}\theta)= 0\,,
\label{ineq:step-2}
\end{eqnarray}
by using elementary calculus and the equality
$\cA^{-1}\theta_t= -\theta-v_t$ (which is just \eqref{th-eqn}).
Now, a straightforward calculation shows that
\[
\cF_2(v):= F_2(u)-F_2(w)= -(Q-\|\cA^{1/2}u\|^2)\cA v +
(\cA^{1/2}(u+w),\cA^{1/2} v)\cA w\,;
\]
then, in view of the bounds in \eqref{energ-bnd}, we easily get
\[
|(\cF_2(v),\eps v+\cA^{-1}\theta)|\le C_R(\eps_0)\left(\|\cA v\|\, \|v\|
+\|\cA^{1/2} v\|\, \|\theta\|\right)\,,
\]
for any $\eps\in (0,\eps_0]$.
Thus, integrating in time between $0$ and $T$ the equality
\eqref{ineq:step-2} and by choosing appropriately $\eps>0$,
we see that
\bey
\lefteqn{
\int_0^T \left( ||\cA v||^2+ \|M_{\gamma}^{1/2} v_t\|^2 + \|\theta\|^2\right)  dt
}
\nonumber\\[1mm] & &
\le   c_0\left( E^0_{v,\theta}(T)+ E^0_{v,\theta}(0)\right)+
 c_1\int_0^T \|\cA^{1/2}\theta \|^2 dt
\nonumber\\[1mm]
& & \qquad +\,
 4\beta\kappa \Big|\int_0^T (z_t,\eps v+ \cA^{-1}\theta)_{\Gamma_0}\,dt\Big|
+ C(R)\int_0^T \|v \|^2 dt\,,
\label{potential-for-v}
\eey
where we have set
\[
E^0_{v,\theta}(t)=E^0_{v}(t)+\frac12\|\theta(t)\|^2\,.
\]
Integrating by parts in time and using the standard form
of the trace theorem it is easy to see that
\[
\Big|\int_0^T (z_t,\eps v+\cA^{-1}\theta)_{\Gamma_0}\,dt\Big|
\le  C_1\left( E^0(T)+ E^0(0)\right)+
\Big|\int_0^T (z,\eps v_t+\cA^{-1}\theta_t)_{\Gamma_0}\,dt\Big|,
\]
where $E^0(t)$ is given by \eqref{e0-full}.
Therefore, using once more the equality
$\cA^{-1}\theta_t=-\theta-v_t$ we obtain
\bey
\lefteqn{
\Big|\int_0^T (z_t,\eps v+\cA^{-1}\theta)_{\Gamma_0}\,dt\Big|
\le
 C_1\left( E^0(T)+ E^0(0)\right)+
\Big|\int_0^T (z,(\eps-1) v_t-\theta)_{\Gamma_0}\,dt\Big|}
\nonumber
\\ [1mm]
& &  \qquad \le C_1\left( E^0(T)+ E^0(0)\right)+
\eps_1 \int_0^T\|A^{1/2}z\|^2\,dt
+  \eps_2 \int_0^T\|M_{\gamma}^{1/2}v_t\|^2\,dt
\nonumber \\
& & \myspace+\,C_2 \int_0^T \|\theta\|^2_{\Gamma_0}\,dt
+C_3(\eps_1,\eps_2) \int_0^T \|z\|^2_{\Om}\,dt
\label{ineq:inner-product}
\eey
for every $\eps_1, \eps_2>0$.
Using now \eqref{ineq:inner-product} in \eqref{potential-for-v} shows
\bey
\int_0^T E^0_{v,\theta}(t)  dt
\le   c_0\left( E^0(T)+ E^0(0)\right)+
c_1 \int_0^T \|\cA^{1/2}\theta \|^2 dt
\nonumber\\[1mm]
 \qquad +\,
\eps \int_0^T \| A^{1/2} z \|^2 dt
+ C(R,\eps )\int_0^T\left(\|  z \|^2 +\|v \|^2\right) dt \,.
\label{potential-for-v1}
\eey
Combined with the estimate \eqref{potential-for-z}, \eqref{potential-for-v1}
establishes
\begin{eqnarray}
\int_{0}^{T} E^0(t) dt
& \le &  c_0\left[ E^0(T)+ E^0(0)\right]+
c_1 \int_0^T\left( \|z_t\|^2+ \|\cA^{1/2}\theta\|^2\right) dt
\nonumber \\
& &
 +\,  c_2  H_0^T(z)+
c_3(R) \int_0^T\left( \|z\|^2+ \|v\|^2\right) dt\,.
\label{potential-for-zv}
\end{eqnarray}
On the other hand, it follows from Lemma~\ref{le:ener-z} that
\bey \label{eq:before-last}
E^0(0)= E^0(T) +\beta G_0^T(z)+ \alpha  \int_0^T \|\cA^{1/2}\theta\|^2 dt
\\
+ \beta\int_0^T (\cF_1(z),z_t)\,d\tau
+\alpha \int_0^T (\cF_2(v),v_t)\,d\tau\,,\nonumber
\eey
and that
\be\label{ineq:last}
T E^0(T)\le \int_0^T E^0(t)dt
- \beta \int_0^T \int_t^T (\cF_1(z),z_t)\,d\tau\,dt
- \alpha \int_0^T \int_t^T (\cF_2(v),v_t)\,d\tau\,dt\,.
\ee
Therefore, combining \eqref{ineq:last} with \eqref{potential-for-zv}
and \eqref{eq:before-last}, we see that \eqref{ineq:main}
holds true, provided that $T$ is sufficiently large.
This concludes the proof of Proposition~\ref{pr:main}.
\qed


\section{Asymptotic smoothness and proof of Theorem~\ref{t:main}}
\label{s:as-smooth}
This section is mainly focused on the asymptotic compactness of the semi-flow
$S_t$ generated by the PDE system \eqref{pde-system}.
In fact, this property is central to the proof of existence of global
attractors; see \cite{babin,chueshov-book,hale,Tem97}.
There are many variants of the notion of asymptotic compactness
in the literature.
In the present case it is convenient to use the formulation introduced
by J.~Hale; see, e.g., \cite{hale} and the references therein.
We recall from \cite{hale} that a dynamical system
$(X,S_t)$ is said to be {\em asymptotically smooth}
iff for any bounded set $\cB$ in $X$ such that $S_t\cB\subset \cB$
for $t>0$ there exists a compact set $\cK$ in the closure $\overline{\cB}$
of $\cB$, such that
\[
\lim_{t\to +\infty}\sup_{y\in \cB} {\rm dist}_X \{S_ty, \cK\}=0\,.
\]
In order to establish this key property, we aim to apply a
compactness criterion due to \cite[Thm.~2]{khan}.
This result is recorded below in the abstract formulation given and used in
\cite{chu-las-jde06} (see also \cite[Chap.2]{mem}).

\begin{proposition}[\cite{chu-las-jde06}] \label{prop:khan}
Let $(X,S_t)$ be a dynamical system on a complete metric space $X$ endowed
with a metric  $d$.
Assume that for any bounded positively invariant set $B$ in $X$ and for any
$\epsilon>0$  there exists $T=T(\epsilon,B)$ such that
\[
d(S_Ty_1,S_Ty_2)\le \epsilon + \Psi_{\epsilon,B,T}(y_1,y_2)\,,
\qquad y_i\in B\,,
\]
where $\Psi_{\epsilon,B,T}(y_1,y_2)$ is a nonnegative function defined
on $B\times B$ such that
\be\label{sequential-limit}
\liminf_{m\to \infty} \liminf_{n\to \infty}\Psi_{\epsilon,B,T}(y_n,y_m)=0
\ee
for every sequence $\{y_n\}_n$ in $B$.
Then the dynamical system $(X,S_t)$ is asymptotically smooth.
\end{proposition}

The main result in this section is the following assertion.

\begin{theorem} \label{thm:asympt-smooth}
Let Assumption~\ref{hypo-1} hold.
Then the dynamical system $(Y,S_t)$ generated by the PDE problem
\eqref{pde-system} is asymptotically smooth.
\end{theorem}

\proof
Since any bounded positively invariant set belongs to $\cW_R$ for some $R>R_*$,
where $\cW_R$ is defined by (\ref{wr-set}), it is sufficient to
consider the case $B=\cW_R$ for every $R>R_*$ only.
Let $y_1, y_2\in \cW_R$. The notation used below is the same as in
Proposition~\ref{pr:main}. Namely, we denote the solutions corresponding
to initial data $y_1$ and $y_2$, respectively, by
$$
(h(t),h_t(t),u(t),u_t(t),\psi(t)):=S_t y_1\,,\quad
(\zeta(t),\zeta_t(t),w(t),w_t(t),\xi(t)):=S_t y_2\,,
$$
and define
$z(t):=h(t)-\zeta(t)$, $v(t):=u(t)-w(t)$, and $\theta(t):=\psi(t)-\xi(t)$.
We seek to establish an estimate of $d(S_Ty_1,S_Ty_2)$ as required
by Proposition~\ref{prop:khan}; this will eventually enable us to
achieve the conclusion of Theorem~\ref{thm:asympt-smooth}.
To accomplish this goal, a major role is played by the following result.

\begin{proposition} \label{prop:energy-in-T}
Let the assumptions of Theorem~\ref{thm:asympt-smooth} be in force.
Then, given $\epsilon>0$ and $T>1$ there exists constants $C_\epsilon(R)$
and $C(R,T)$ such that
\bey
E^0(T) \le \epsilon + \frac{1}{T}\,
\big[\, C_\epsilon(R)+ \Psi_T(z,v) \big] + C(R,T) \,lot(z,v)\,,
\label{ineq:energy-in-T}
\eey
where $E^0(t):=\beta E^0_z(t)+\alpha E^0_v(t)+\frac{\alpha}2\|\theta(t)\|^2$ with
 $E_z^0(t)$ and $E_v^0(t)$ given in \eqref{zero-energies},
the functional $\Psi_T(z,v)$ is given by \eqref{psi-def} while
$lot(z,v)$ is defined by
\[
lot(z,v):= \sup_{[0,T]} \|z(t)\|
+ \sup_{[0,T]}\| v(t)\|\,.
\]
\end{proposition}

\proof
To establish \eqref{ineq:energy-in-T}, we return to the main inequality
\eqref{ineq:main} and proceed with the estimate of its right hand side.
Preliminarly, let us record a useful inequality, which holds under the only
Assumption~\ref{hypo-0}; see, e.g., \cite[Lemma~5.3]{bcl-cpaa-06}, or
\cite[Section~5.3]{mem}.
There exists a constant $C_0>0$ such that
\be \label{ineq:difficult-g}
|(D(\zeta+z)-D(\zeta),h)|
\le  C_{0}\,\big[(D(\zeta),\zeta)+ (D(\zeta+z),\zeta+z)\big]\,
\|A^{1/2}h\|+ C_{0}\| h\|
\ee
for any $\zeta, z, h\in \cD(A^{1/2})$, where $D$ is the damping operator
given by \eqref{Dh-mem}.
In view of \eqref{ineq:difficult-g}, $H_s^t(z)$ defined by
\eqref{integrals-st-z} satisfies
\[
H^T_0(z)\le C_0\int_0^T\left[(D(\zeta_t),\zeta_t)
+ (D(h_t),h_t)\right] \|A^{1/2}z\| dt +C_1 \int_0^T \|z\|\,dt\,.
\]
Moreover, from the energy inequality \eqref{energy-ineq} we know
\begin{subequations}\label{bounds}
\bey
& &
\beta \int_0^t (D(h_t),h_t)dt
+ \alpha \int_0^t \|\cA^{1/2}\psi\|^2 dt \le C_{R}\,,
\\
& & \beta \int_0^t (D(\zeta_t),\zeta_t) dt
+ \alpha  \int_0^t \|\cA^{1/2}\xi\|^2 dt \le C_{R}\,,
\eey
\end{subequations}
where crucially $C_{R}$ does not depend on $t$.
Then, the estimates \eqref{bounds} combined with the fact that
$\|A^{1/2}z(t)\|\le C_R$ for all $t\in[0,T]$ show that
\begin{equation}\label{h0T}
H_0^T(z)\le C_R + C\,T\,lot(z,v)\,.
\end{equation}
Next, using now the inequality in Assumption~\ref{hypo-1}
and once again the uniform estimates \eqref{bounds}, one can see that
\begin{equation}\label{g0T}
\int_0^T\left( \|z_t\|^2  + \|\cA^{1/2}\theta\|^2 \right)dt
\le \eps T + C_\eps(R)\quad\mbox{for every}\; \eps>0\,.
\end{equation}
On the other hand, taking $t=0$ in \eqref{key-equality} and using the
fact that $E^0(0)\le C_R$, we get
\be \label{ineq-for-timezero}
\beta  G_0^T(z) + \alpha  \int_0^T\|\cA^{1/2}\theta\|^2 dt
\le C_R+\beta\Big|\int_0^T (\cF_1(z),z_t)\,d\tau\Big|
+ \alpha\Big|\int_0^T (\cF_2(v),v_t)\,d\tau\Big|\,.
\ee
Therefore, \eqref{ineq:energy-in-T} follows from \eqref{ineq:main}
of Proposition~\ref{pr:main}
combining the estimates \eqref{h0T}, \eqref{g0T} and \eqref{ineq-for-timezero}.
\qed

\medskip
We are now in a position to complete the proof
of Theorem~\ref{thm:asympt-smooth}.
It follows from Proposition~\ref{prop:energy-in-T} that given $\epsilon>0$
there exists $T=T(\epsilon)>1$ such that for initial data $y_1, y_2\in B$
we have
\begin{eqnarray*}
||S_Ty_1-S_Ty_2||_Y & = &||(z(T),z_t(T),v(T),v_t(T),\theta(T))||_Y
\\
& \le & C |E^0(T)|^{1/2}
\le   \epsilon + \Psi_{\epsilon,B,T}(y_1,y_2)\,,
\end{eqnarray*}
where
\be\label{psi}
\Psi_{\epsilon,B,T}(y_1,y_2)=
C_{B,\epsilon,T}\big\{ \Psi_{T}(z,v)+lot(z,v)\big\}^{1/2}\,
\ee
with $ \Psi_{T}(z,v)$ is given by \eqref{psi-def}.
Thus, since we aim to invoke Proposition~\ref{prop:khan}, what we need to
show is the validity of the sequential limits
\eqref{sequential-limit} for $\Psi_{\epsilon,B,T}$ defined by \eqref{psi}.
This requires pretty much the same arguments used
in \cite{bcl-cpaa-06}
(see also \cite{snowbird} for the case of a pure wave equation and
\cite{chu-las-jde06} for the case of von Karman plates).

\medskip
\noindent
{\em Proof of Theorem~\ref{t:main}. }
Since $\cW_R$ is bounded and positively invariant,
Theorem~\ref{thm:asympt-smooth} implies the existence of
a compact global attractor $\mfA_R$ for the dynamical system
$(\cW_R,S_t)$, for each $R\ge R_*$.
Choose now $R_0\ge R_*+1$ such that the set $\cN$ of equilibria lies
in $\cW_{R_0-1}$.
By \eqref{lip-g2} we have that $sg(s)>0$ for $s\neq 0$.
Therefore the energy inequality (\ref{energy-ineq}) implies that
the energy $\cE$ given by \eqref{total-energy} is a strict Lyapunov
function for $(\cW_R,S_t)$.
This, in turn, implies (see, e.g., \cite[Theorem 3.2.1]{babin} or
\cite[Theorem 1.6.1]{chueshov-book}) that $\mfA_R=M^u(\cN)$ and
by Lemma \ref{lemma:b-stat} $\mfA_R$ does not depend on $R$ for $R\ge R_0$;
moreover, relation (\ref{stb-to-n}) holds.

We finally observe that since the energy $\cE$ is non increasing along the
trajectories and as the attractor has the structure $\mfA=M^u(\cN)$ and $\cN\subset\mfA$, we have
that
\[
\textrm{sup}\, \big\{ \cE(U): \; U\in \mfA\big \}
= \textrm{sup}\,\big\{ \cE(U): \; U\in \cN\big\}\,.
\]
This implies, in particular, that the global attractor is such that
\[
\mfA \subseteq \big\{ U\in Y: \; \|U\|_Y\le R\big\}\,,
\]
where $R$ does not depend on $\gamma$ and $\kappa$. Therefore,
\eqref{first-bnd} holds, and the proof of Theorem \ref{t:main} is
completed.
\qed


\section{Stabilizability estimate}\label{s:stab-est}
In this section we derive a {\em stabilizability estimate} which
will play a key role in the proofs of both finite-dimensionality
and regularity of attractors.
In fact, it will enable us to apply some abstract results presented in
\cite{chu-las-1} and \cite{mem} which are central to the proof of
Theorem~\ref{t:main2}.

\begin{proposition}[Stabilizability estimate]  \label{p:stabiliz}
Let Assumption~\ref{hypo-2} hold.
Then there exist positive constants $C_1$, $C_2$ and $\om$ depending on $R$
but independent of $\gamma,\kappa\in [0,1]$
such that, for any $y_1, y_2\in\cW_R$, the following estimate holds true:
\begin{equation}\label{stab-est}
\| S_ty_1- S_t y_2\|^2_Y\le C_1 e^{-\om t}\| y_1- y_2 \|^2_Y +
C_2\, lot_t(h-\zeta,u-w)\,,\quad  t>0\,,
\end{equation}
where
\begin{equation}\label{lot-t-zv}
lot_t(z,v)=  \max_{[0,t]}
\left( \|z(\tau)\|^2_{1-\delta,\Om}+ \|v(\tau)\|^2_{2-\delta,\Gamma_0}\right).
\end{equation}
Above, we have used the notation
$$
(h(t),h_t(t),u(t),u_t(t),\psi(t)):=S_t y_1\,,\quad
(\zeta(t),\zeta_t(t),w(t),w_t(t),\xi(t)):=S_t y_2\,.
$$
\end{proposition}

\proof
Given $y_1, y_2\in\cW_R$, let $S_t y_1=(h,h_t,u,u_t,\psi)$ and
$S_t y_2=(\zeta,\zeta_t,w,w_t,\xi)$ be the corresponding
solutions, as introduced in Proposition~\ref{pr:main}.
Moreover, let $z=h-\zeta$, $v=u-w$, $\theta=\psi-\xi$, as originally defined
in \eqref{zvth-dif}.
We recall that for these solutions the bounds \eqref{energ-bnd},
\eqref{energ-bnd2} and \eqref{bounds} hold true.
Since we seek to establish an estimate of $\|S_ty_1- S_t y_2\|^2_Y$,
our starting point will be once again the fundamental inequality
\eqref{ineq:main} pertaining to the energy
$E^0(t)$ given by \eqref{e0-full}.
Then, we need to produce accurate bounds for the various terms which occur
in the right hand side of \eqref{ineq:main}.
It will be used throughout that, due to  \eqref{energ-bnd}
and \eqref{energ-bnd2}, there exists a constant $C_R$ such that
$\|S_ty_i\|_Y\le C_R$, $i=1,2$.

The following Lemma provides two key estimates pertaining to the nonlinear
terms $\cF_1$, $\cF_2$.
\begin{lemma}\label{le:f12}
Under Assumption~\ref{hypo-2}, the following estimates hold true
for some $\delta>0$:
\begin{subequations}
\label{cf}
\bey
\lefteqn{\hspace{-8mm}
\Big|\int_t^T (\cF_1(z),z_t)\,d\tau \Big|
\le \eps \int_0^T  \|A^{1/2}z\|^2_{\Om}\, d\tau
}
\nonumber
\\[2mm]
& &  
+ \,C_\eps(R) \int_0^T \left( \|h_t\|^2_{\Om} +
\|\zeta_t\|^2_{\Om}\right) \|A^{1/2}z\|^2_{\Om}\, d\tau
+ \,C_{R,T}\max_{[0,T]} \|z\|^2_{1-\delta,\Om}\,,
\label{cf-1}
\\[3mm]
\lefteqn{\hspace{-8mm}
\Big|\int_t^T (\cF_2(v),v_t)\,d\tau\Big|
\le  \eps\int_0^T \left(\|\cA v\|^2 + \|v_t\|^2\right)\,d\tau
}
\nonumber
\\[2mm]
& & 
+\, C_R \int_0^T\left( \|\cA v \|^2 + \|v_t\|^2\right)\|\xi\|\,d\tau
+ \,C_\eps(R,T)\max_{[0,T]}\|v\|^2_{1,\Gamma_0}\,,
\label{cf-2}
\eey
\end{subequations}
for all $t\in [0,T]$,
where $\eps>0$ can be taken arbitrarily small.
Here, $\cF_1$ and $\cF_2$ are given by \eqref{cal-f}.
\end{lemma}

\begin{proof}
The former estimate has been established in \cite{bcl-cpaa-06} in the 
isothermal case;
since the arguments needed in the present case are pretty much the 
same, the proof of \eqref{cf-1} is omitted. 
To show \eqref{cf-2}, it is convenient to recall the expression of 
the nonlinear term $F_2$ acting on the plate component, that is
\begin{equation} \label{def:force-2}
F_2(v)= -\big(Q-||\cA^{1/2}v||^2\big)\cA v - p_0\,,
\end{equation}
and that $F_2(v)=\Phi'(v)$, with $\Phi(v)$ defined by \eqref{plate-potential}.
So then, we write explicitly
\begin{eqnarray*}
(\cF_2(v),v_t) &=& (F_2(w+v),v_t)-(F_2(w),v_t)
\\[1mm]
& = &\frac{d}{dt}\Big[\Phi(u)-(F_2(w),v)\Big] + (F_2'(w)w_t,v)-(F_2(u),w_t)\,
\end{eqnarray*}
with $u=w+v$,
which yields, adding and subtracting $(F_2(w),w_t)$, the following
initial decomposition:
\[
(\cF_2(v),v_t) = \frac{d}{dt}Q_0(v) + P_0(v)\,,
\]
with
\begin{subequations}
\begin{eqnarray}
Q_0(v) &=&\Phi(u) -\Phi(w)-(F_2(w),v)\,,
\\[2mm]
P_0(v) &=&(F_2'(w)w_t,v)-(F_2(u)-F_2(w),w_t)\,.
\label{def:p-0}
\end{eqnarray}
\end{subequations}
By inserting the explicit expression \eqref{def:force-2} of $F_2$
we rewrite $Q_0(v)$ as follows:
\begin{eqnarray*}
Q_0(v)&=& \int_0^1 \big[(F_2(w+\lambda v),v) - (F_2(w),v)\big]\, d\lambda
\nonumber
\\ [1mm]
&=&-\int_0^1 \lambda\big(Q- \|\cA^{1/2}(w+\lambda v)\|^2\big)
(\cA v,v)\, d\lambda
\nonumber
\\ [1mm]
& & \qquad + \int_0^1 \big(\|\cA^{1/2}(w+\lambda v)\|^2-\|\cA^{1/2}w\|^2\big)
(\cA w,v)\, d\lambda\,,
\end{eqnarray*}
which immediately gives
\begin{equation} \label{estimate:q0}
|Q_0(v)| \le C_R \|\cA^{1/2}v\|^2\,,
\end{equation}
in view of \eqref{energ-bnd}.
Let us turn to the analysis of the term $P_0$ as defined by \eqref{def:p-0}.
We preliminarly compute the first and second Fr\'echet derivatives of the
nonlinear function $F_2$ in \eqref{def:force-2}:
\[
F_2'(w)v= -\big(Q-||\cA^{1/2}w||^2\big)\cA v + 2 (\cA w,v)\cA w\,,
\]
\[
\langle F_2''(w);v,\tilde{v}\rangle
= \frac{d}{d\eta}F_2'(w+\eta\tilde{v})v\Big|_{\eta=0}
=2 (\cA \tilde{v},w)\cA v + 2 (\cA \tilde{v},v)\cA w
+ 2 (\cA w,v)\cA\tilde{v}\,;
\]
in particular, we have that
\[
\langle F_2''(w);v,v\rangle= 4 (\cA w,v)\cA v + 2 \|\cA^{1/2} v\|^2\cA w\,.
\]

Observing initially that $F_2'(w)$ is a symmetric operator, we rewrite
\eqref{def:p-0} as follows,
\begin{eqnarray*}
P_0(v)&=&
\int_0^1 \big(\big[(F_2'(w)- (F_2'(w+\lambda v)\big]v,w_t\big)\,d\lambda
\nonumber\\ [1mm]
&=&-\int_0^1 \lambda\,d\lambda
\int_0^1( \langle F_2''(w+\lambda\eta v);v,v\rangle,w_t)\,d\eta
\nonumber\\ [1mm]
&=&-\int_0^1 \lambda\int_0^1 4(\cA[w+\eta\lambda v],v)(\cA v,w_t)
\,d\eta\,d\lambda
\nonumber\\ [1mm]
& & \quad -\int_0^1 \lambda\int_0^1 2\|\cA^{1/2}v\|^2\,(\cA[w+\eta\lambda v],w_t)
\,d\eta\,d\lambda\,.
\end{eqnarray*}

We now substitute $w_t=-\cA^{-1}\xi_t-\xi$ in the above formula,
thus obtaining
\[
P_0(v)=  P_{01}(v)+P_{02}(v)\,,
\]
where
\begin{eqnarray*}
\lefteqn{\hspace{-1cm}
P_{01}(v) = \,4\int_0^1 \lambda\int_0^1 (\cA[w+\eta\lambda v],v)
(\cA v,\xi)\,d\eta\,d\lambda }
\\[1mm]
& & \qquad +\,2 \int_0^1 \lambda\int_0^1 \|\cA^{1/2}v\|^2\,
(\cA[w+\eta\lambda v],\xi) \,d\eta\,d\lambda\,,
\\[2mm]
\lefteqn{\hspace{-1cm}
P_{02}(v)=\underbrace{4\int_0^1 \lambda\int_0^1 (\cA[w+\eta\lambda v],v)
(v,\xi_t)\,d\eta\,d\lambda}_{P_{021}(v)} }
\\[1mm]
& & \qquad +\underbrace{2 \int_0^1 \lambda\int_0^1 \|\cA^{1/2}v\|^2\,
(w+\eta\lambda v,\xi_t)
\,d\eta\,d\lambda}_{P_{022}(v)}\,.
\end{eqnarray*}
Then it is straightforward to see that
\begin{equation}
\begin{array}{lcl}
|P_{01}(v)|
&\le & C_R\, \big[\,|(\cA v,\xi)|\,\|v\| + \|\cA^{1/2}v\|^2\|\xi\|\,\big]
\le C_R \|v\|\,\|\cA v\|\,\|\xi\|
\label{estimate:p0}
\\[2mm]
&\le & \epsilon \|\cA v\|^2+ C_{R,\epsilon}\|v\|^2\quad\mbox{for every $\epsilon>0$}.
\end{array}
\end{equation}
Instead, the terms $P_{02i}(v)$ require further splitting.
It is elementary to see that
\begin{eqnarray*}
\lefteqn{
P_{021}(v)=4\int_0^1 \lambda d\lambda \int_0^1 d\eta
\frac{d}{dt}\Big[(\cA[w+\lambda\eta v],v)(v,\xi)\Big] }
\\[1mm]
& & \qquad - 4\int_0^1 \lambda d\lambda \int_0^1 d\eta
(w_t+\lambda\eta v_t,\cA v)(v,\xi)
\\[1mm]
& &\qquad  - 4\int_0^1 \lambda d\lambda \int_0^1 d\eta
(\cA[w+\lambda\eta v],v_t)(v,\xi)
\\[1mm]
& & \qquad - 4\int_0^1 \lambda d\lambda \int_0^1 d\eta
(\cA[w+\lambda\eta v],v)(v_t,\xi)=:\frac{d}{dt}Q_1(v)
+\sum_{j=1}^3 P_{021}^j(v)\,.
\end{eqnarray*}
Readily
\begin{eqnarray}
& & |Q_1(v)| \le C_R\, \|v\|^2 \|\xi\|
\le C_R \|v\|^2\,;
\label{estimate:q1}\\[2mm]
& & |P_{021}^1(v)| \le  C_R\, \|\cA v\| \|v\|\,, \qquad
|P_{021}^2(v)|\,, \; |P_{021}^3(v)| \le  C_R\, \|v\|\|v_t\| \,.
\nonumber
\end{eqnarray}
Therefore the term
$P_{1}(v):=\sum_{j=1}^3 P_{021}^j(v)$ admits the estimate
\begin{equation}
|P_{1}(v)| \le \epsilon \Big(\|\cA v\|^2+\|v_t\|^2\Big)
+ C_{R,\epsilon}\,\|v\|^2\,.
\label{estimate:p1}
\end{equation}
On the other hand, we also have
\begin{eqnarray*}
\lefteqn{
P_{022}(v)= 2 \int_0^1 \lambda d\lambda \int_0^1 d\eta
\frac{d}{dt}\Big[\|\cA^{1/2}v\|^2\,(w+\eta\lambda v,\xi)\Big] }
\\[1mm]
& & \qquad - 2 \int_0^1 \lambda d\lambda \int_0^1 d\eta
\|\cA^{1/2}v\|^2\,(w_t+\eta\lambda v_t,\xi)
\\[1mm]
& & \qquad - 4 \int_0^1 \lambda d\lambda \int_0^1 d\eta
(v_t,\cA v)(w+\eta\lambda v,\xi)=: \frac{d}{dt}Q_2(v)
+ \sum_{j=1}^2 P_{022}^j(v),
\end{eqnarray*}
where now
\begin{eqnarray}
& & |Q_2(v)| \le C_R\, \|\cA^{1/2}v\|^2
\label{estimate:q2}\\[2mm]
& & |P_{022}^1(v)| \le  C_R\, \|\cA^{1/2} v\|^2\,,\;
|P_{022}^2(v)| \le  C_R\,\big(\|\cA v\|^2+\|v_t\|^2\big)\|\xi\|\,,
\nonumber
\end{eqnarray}
so that for $P_{2}(v):=P_{022}^1(v)+P_{022}^2(v)$ we obtain
\begin{equation}
|P_{2}(v)|\le C_R\, \|\cA^{1/2} v\|^2 +
C_R\Big(\|\cA v\|^2+\|v_t\|^2\Big)\|\xi\|\,.
\label{estimate:p2}
\end{equation}

We have therefore proved that the following decomposition holds true,
\begin{equation}\label{final-decompostion}
(\cF_2(v),v_t) = \frac{d}{dt}Q(v) + P(v)\,,
\end{equation}
where
\begin{eqnarray*}
Q(v) &=& Q_0(v)+Q_1(v)+Q_2(v),
\\[2mm]
P(v) &=& P_{01}(v)+P_1(v)+P_2(v)
\end{eqnarray*}
with the bounds \eqref{estimate:q0}, \eqref{estimate:q1},
\eqref{estimate:q2} pertaining to the terms $Q_i$ and
\eqref{estimate:p0}, \eqref{estimate:p1}, \eqref{estimate:p2}
for the terms $P_i$, respectively.
Thus, integrating in time between $t$ and $T$ both sides of
\eqref{final-decompostion} and taking into account the complex of
estimates established above, we finally obtain \eqref{cf-2}.
\end{proof}


Notice now that by the lower bound in \eqref{g-for-dim0} it follows that
$ms^2\le sg(s)$ and thus
\begin{equation}\label{h-zeta-bnd}
\|h_t(t)\|^2_{\Om} +
\|\zeta_t(t)\|^2_{\Om}\le
D_{h,\zeta}(t):= m^{-1}\big[ (D(h_t(t)),h_t(t))_{\Omega}
+ (D(\zeta_t(t)),\zeta_t(t))_{\Om}\big]
\end{equation}
for all $t\ge 0$.
Hence, by using the estimates established in Lemma~\ref{le:f12}, and the
elementary inequality $||\xi||\le \epsilon + (4\epsilon)^{-1}||\xi||^2$,
valid for arbitrary small $\eps>0$, we obtain
the following bound for $\Psi_T(z,v)$ defined by \eqref{psi-def}:
\begin{equation}\label{psi-T}
\Psi_T(z,v)\le \eps \int_0^T E^0(t)dt +C_\eps(T, R)\, \Xi_T(z,v)\,,
\end{equation}
with $E^0(t)$ defined in \eqref{e0-full} and $\Xi_T(z,v)$ given by
\begin{eqnarray}\label{Xi-T}
\Xi_T(z,v)& := & lot_T(z,v)
+ \int_0^T D_{h,\zeta}(\tau) \|A^{1/2}z(\tau)\|^2_{\Om} d\tau
\nonumber \\
& &
+\int_0^T \|\xi(\tau)\|^2\left( \|\cA v(\tau)\|^2_{\Gamma_0}+
\|\cM^{1/2} v_t(\tau)\|^2_{\Gamma_0}\right) d\tau\,.
\end{eqnarray}
Above, we used the notation $lot_T(z,v)$ as in \eqref{lot-t-zv},
while $D_{h,\zeta}(t)$ has been introduced within \eqref{h-zeta-bnd}.

\smallskip
To proceed, we will use the next assertion.
\begin{lemma}\label{le:h0-g0}
Under Assumption~\ref{hypo-2}, the following estimate holds
for $H_0^T(z)$ defined in \eqref{integrals-st-z},
with arbitrarily small $\eps>0$:
\begin{equation}\label{h0-g0}
H^T_0(z)\le \eps \int_0^T E^0(t)\,dt +
C_\eps \left[ \Xi_T(z,v) + G_0^T(z)\right]\,,
\end{equation}
where $G_0^T(z)$ is defined in \eqref{integrals-st},
and $\Xi_T(z,v)$ is given by \eqref{Xi-T}.
\end{lemma}

\begin{proof}
In view of Assumption~\ref{hypo-2} it can be shown that
for every $\epsilon>0$ there exists $C_\epsilon>0$ such that
\bey
\lefteqn{\hspace{-1cm}
|(D(\zeta+z)-D(\zeta),h)|\le C_{\epsilon}\,(D(\zeta+z)-D(\zeta),z) }
\nonumber\\[2mm]
& & \quad + \epsilon \, \big(1+ (D(\zeta),\zeta)+ (D(\zeta+z),\zeta+z)\big)\,
||A^{1/2}h||^2\,,
\label{ineq:difficult-g1}
\eey
for any $\zeta, z, h\in \cD(A^{1/2})$, where $D$ is
the damping operator defined by \eqref{Dh-mem};
see \cite[Section~4]{chu-las-1} or \cite[Chapter~5]{mem}).
Owing to estimate \eqref{ineq:difficult-g1}, it is immediately seen that
\[
H^T_0(z)\le C_\epsilon G_0^T(z)+\epsilon \int_0^T E^0(t)\,dt +
\epsilon m\,\int_0^T D_{h,\zeta}(\tau) \|A^{1/2}z(\tau)\|^2_{\Om} d\tau\,,
\]
which in turn implies \eqref{h0-g0}, as desired.
\end{proof}

Notice that by the lower bound in \eqref{g-for-dim0} one has, as well, that
\begin{equation}\label{zv2-gg}
\int_0^T \|z_t\|^2\,  dt\le \frac1m\, G_0^T(z)\,.
\end{equation}
Thus, we return to the basic inequality \eqref{ineq:main} in
Proposition~\ref{pr:main} and apply the estimates \eqref{zv2-gg},
\eqref{psi-T} and \eqref{h0-g0}. Choosing $\eps$ sufficiently
small, we obtain
\begin{equation}\label{ineq-main2}
TE^0(T) + \int_{0}^{T} E^0(t) \,dt
\le C_1\Big[ \beta G_0^T(z)+ \alpha \int_0^T \|\cA^{1/2}\theta\|^2\Big]
+  C_2(T,R) \, \Xi_T(z,v)
 \end{equation}
for $T\ge T_0>0$.
On the other hand, using the energy equality \eqref{key-equality} and
Lemma~\ref{le:f12} we also have that
\[
\beta  G_0^T(z) + \alpha \int_0^T \|\cA^{1/2}\theta\|^2
\le E^0(0) - E^0(T) + \epsilon  \int_0^T E^0(t)dt
+ C_{\epsilon}(T, R)\, \Xi_T(z,v)
\]
for any $\epsilon>0$.
Combining this bound with \eqref{ineq-main2}, we see that there exists $T>1$
such that
\begin{equation}\label{1-iter}
E^0(T)\le q E^0(0) + C_{R,T}\,\Xi_T(z,v)\quad\mbox{with}\quad
0<q\equiv q_{T,R}<1\,.
\end{equation}
We can now apply the same procedure described in \cite{chu-las-1}
(see also \cite[Chap.~3]{mem}, \cite{bcl-cpaa-06}).
From \eqref{1-iter} we have that
\begin{eqnarray*}
E^0((m+1)T) \le q E^0(mT) + c b_m\,, \quad  m=0,1,2,\ldots,
\end{eqnarray*}
where $c=c_{R,T}>0$ is a constant and
\begin{eqnarray*}
b_m & := & \sup_{t\in [mT, (m+1)T]}
\left( \|z(t)\|^2_{1-\delta}+ \|v(t)\|^2_{2-\delta}\right)
\\ & &
+ \int_{mT}^{(m+1)T} \big[D_{h,\zeta}(\tau)+ \|\cA^{1/2}\xi(\tau)\|^2\big]
E^0(\tau)\, d\tau\
\end{eqnarray*}
with $D_{h,\zeta}(t)$  defined in \eqref{h-zeta-bnd}.
This  yields
\[
E^0(mT) \leq q^m E^0(0) + c \sum_{l =1}^m q^{m -l} b_{l-1}.
\]
Since $q < 1$,  there exists $\omega > 0 $ such that
\begin{eqnarray*}
E^0(mT) & \leq & C_1 e^{-\omega mT}E^0(0)
\\ & & + C_{2}
\Big\{ lot_{mT}(z,v)
+\int_0^{mT}\!\!\! e ^{-\omega (mT-\tau)}\left[ D_{h,\zeta}(\tau)+ \|\cA^{1/2}\xi\|^2\right]  E^0(\tau)\Big\},
\end{eqnarray*}
which  implies that
$$
E^0(t)\le C_1 e^{-\omega t}E^0(0) + C_{2}
\Big\{ lot_{t}(z,v)
 +\int_0^{t} e ^{-\omega (t-\tau) }\left[ D_{h,\zeta}(\tau)+ \|\cA^{1/2}\xi\|^2\right]  E^0(\tau) d\tau \Big\}
$$
for all $t\ge0$.
Therefore, applying Gronwall's lemma we find that
\[
E^0(t)\le \big[ C_1 E^0(0)e^{-\om t} + C_{2}\, lot_{t}(z,v)\big]
 \,\exp\Big\{
C_2 \int_0^{t} \left[ D_{h,\zeta}(\tau)+ \|\cA^{1/2}\xi\|^2\right]  d\tau \Big\}.
\]
Since, by \eqref{bounds} we have that
\[
\int_0^{t} \left[ D_{h,\zeta}(\tau)+ \|\cA^{1/2}\xi\|^2\right] \ d\tau
\le C_R,\quad\mbox{for all}\quad t\ge 0,
\]
we obtain the estimate \eqref{stab-est}.
This concludes the proof of Theorem~\ref{p:stabiliz}.
\qed

\begin{remark}
\begin{rm}
It is important to emphasize that the stabilizability estimate
\eqref{stab-est} is independent of the parameters $\gamma$ and $\kappa$.
This fact will be crucially important in the proof of Theorem~\ref{co:up-sc}.
\end{rm}
\end{remark}

\section{Proof of Theorem~\ref{t:main2} and  Theorem~\ref{co:up-sc}}\label{s:proof-t2}
{\em $1$. Finiteness of fractal dimension.}
To prove finiteness of the fractal dimension $\dim_f\mfA$,
we appeal to a generalization of Ladyzhenskaya's theorem on dimension
of invariant sets, given by \cite[Theorem~2.10]{chu-las-1}.
This result applies, in view of the local Lipschitz continuity
of the semi-flow $S_t$ (see \eqref{SF}) and of the stabilizability
estimate shown in Proposition~\ref{p:stabiliz}.

Following the method described in \cite{chu-las-1} (subsequently
used in \cite{mem,snowbird,chu-las-preprint06}),
let us introduce the extended space $Y_T=Y\times W_1(0,T)$
(with an appropriate $T>0$), where $W_1(0,T)$ is the space of pairs
$(z(t),v(t))$ in $L_2(0,T;H^1(\Om)\times H^2(\Gamma_0)\cap H^1_0(\Gamma_0))$
such that $(z_t(t),v_t(t))$ belongs to
$L_2(0,T;L_2(\Om)\times L_2(\Gamma_0))$.
Next, we consider in $Y_T$ the set
\[
\mfA_T=\left\{W=(U; z(t,U), v(t,U), t\in [0,T]) :\; U\in \mfA  \right\}\,,
\]
where $(z(t,U),v(t,U),\theta(t,U))$ is a solution to \eqref{abstract-system},
$U=(z^0,z^1,v^0,v^1,\theta^0)$, and define a shift operator
$V: \mfA_T\mapsto\mfA_T$ by the formula
\[
V :\; (U; z(t,U), v(t,U), t\in [0,T])\mapsto
(S_TU; z(t+T,U), v(t+T,U), t\in [0,T])\,.
\]
Then, by using pretty much the same arguments as in \cite[Sect.3.4]{chu-las-1}
(or in \cite[Ch.~4]{mem}, \cite{snowbird}, \cite{chu-las-preprint06}), we see that
the assumptions of Theorem~2.10 in \cite{chu-las-1} are satisfied.

That ${\rm dim}_f\mfA$ is independent of $\gamma$ and $\kappa$
follows from the fact that the estimate \eqref{stab-est} is uniform
with respect to $\gamma$ and $\kappa$.
The analysis is similar to the one carried out in the case of thermoelastic
von Karman evolutions; see \cite{chu-las-preprint06}.

\medskip
\noindent
{\em $2$. Smoothness of the global attractor.}
The additional regularity of the attractor follows from
Theorem~\ref{p:stabiliz}, by using similar arguments
as in \cite[Sect.~4.2]{mem}
(see also \cite{bcl-cpaa-06}, where the isothermal case is studied, or
\cite{chu-las-jde06,snowbird,chu-las-preprint06}, dealing with
different models).

Let $\{ y(t)\equiv (z(t),z_t(t),v(t),v_t(t),\theta(t))\, :\, t\in\R\}\subset Y$
be a full trajectory from the attractor $\mfA$.
Let $|\sigma|<1$.
Applying Theorem~\ref{p:stabiliz} with $y_1=y(s+\sigma)$, $y_2=y(s)$
(and, accordingly, the interval $[s,t]$ in place of $[0,t]$), we have that
\begin{eqnarray}
\lefteqn{
\|y(t+\sigma)- y(t)\|^2_Y \le C_1 e^{-\om (t-s)}\|y(s+\sigma)- y(s)\|^2_Y }
\nonumber \\[2mm]
& & +\, C_2 \max_{\tau\in [s,t]}
\left( \|z(\tau+\sigma)-z(\tau) \|^2_{1-\delta}+
\|v(\tau+\sigma)-v(\tau)\|^2_{2-\delta}\right)
\label{stab-est-atr}
\end{eqnarray}
for any $t,s\in\R$ such that $s\le t$ and for any $\sigma$ with $|\sigma|<1$.
Letting $s\to-\infty$, \eqref{stab-est-atr} gives
\[
\|y(t+\sigma)- y(t)\|^2_Y \le
C_2
 \max_{\tau\in [-\infty,t]}
\left( \|z(\tau+\sigma)-z(\tau) \|^2_{1-\delta}+
\|v(\tau+\sigma)-v(\tau)\|^2_{2-\delta}\right)
\]
for any $t\in\R$  and  $|\sigma|<1$.
By interpolation we have that
\begin{eqnarray*}
\|z(\tau+\sigma)-z(\tau) \|^2_{1-\delta}+
\|v(\tau+\sigma)-v(\tau)\|^2_{2-\delta}\le \eps \|y(t+\sigma)- y(t)\|^2_Y
 \\[2mm]
+\,C_\eps \left( \|z(\tau+\sigma)-z(\tau) \|^2+
\|v(\tau+\sigma)-v(\tau)\|^2\right)
\end{eqnarray*}
for every $\eps>0$. Therefore, we obtain
\begin{equation}\label{Sm}
\max_{\tau\in [-\infty,t]} \|y(\tau+\sigma)- y(\tau)\|^2_Y \le
C\, \max_{\tau\in [-\infty,t]}
\left( \|z(\tau+\sigma)-z(\tau) \|^2+
\|v(\tau+\sigma)-v(\tau)\|^2\right)
\end{equation}
for any $t\in\R$ and  $|\sigma|<1$.
On the other hand, on the attractor we have that
\[
\frac{1}{\sigma} \|z(\tau+\sigma)-z(\tau) \|
\le\frac1\sigma \, \int_0^\sigma \|z_t(\tau+t)\| d\tau\leq C,\quad t \in \R\,,
\]
and
\[
 \frac{1}{\sigma}
\|v(\tau+\sigma)-v(\tau) \|
\le\frac1\sigma \, \int_0^\sigma \|v_t(\tau+t)\| d\tau\leq C,\quad t \in \R\,,
\]
which combined with \eqref{Sm} give
\begin{equation*}
\max_{\tau\in \R}\left\|\frac{y(\tau+\sigma)- y(\tau)}{\sigma}\right\|^2_Y
\le C\quad\mbox{for }\; |\sigma|<1\,.
\end{equation*}
This implies that
\begin{equation}\label{smooth-1}
\| z_{tt}(t)\|^2+ \| A^{1/2}z_{t}(t)\|^2+\| v_{tt}(t)\|^2
+\gamma \|\cA^{1/2} v_{tt}(t)\|^2 + \| \cA v_{t}(t)\|^2 +\|\theta_t(t)\|^2
\le C
\end{equation}
for all $t\in\R$ and for any  trajectory
$\{ y(t)\equiv (z(t),z_t(t),v(t),v_t(t), \theta(t))\, :\, t\in\R\}$
from the attractor $\mfA$.

The estimate \eqref{smooth-1} enables us to establish the spatial smoothness
of the attractor.
Let us begin with the analysis of the plate variable $v$.
In view of \eqref{plate-eq-var}, \eqref{heat-eq-var} and
\eqref{smooth-1} it is easily seen that on the attractor one has
\[
\|\Delta \theta(t)\|\le C\,\big(\|\cA v_t\|+ \|\theta_t\|\big)\le C\,,
\]
and hence, since $\|\cA^{-1/2}M_\gamma^{1/2}\|\le C$,
\[
\|\cA^{-1/2}\Delta^2 v\|\le \|M_\gamma^{-1/2}\Delta^2 v\|
\le C\,\big(\|\Delta\theta\|+\|F_2(v)\|+ \|M_\gamma^{1/2}v_{tt}\| +
\|z_t\|_1\big)\le C\,,
\]
for $\gamma>0$, whereas $\|\Delta^2 v\|\le C$ in the case $\gamma=0$.
Thus, we can conclude that $t\mapsto v(t)$ is a bounded function in
the space $W_\gamma$ defined in \eqref{w-gamma}.

As for the wave component $z$, in the case $n=2$ or $\{n=3, 1\le p\le 3\}$
from \eqref{lip-g1} it follows that
\[
\|g(z_t)\|\le C\big( 1+ \|z_t\|^p_{L_{2p}(\Om)}\big)\le
C\big( 1+ \|z_t\|^p_{1,\Om}\big)\,.
\]
Therefore from \eqref{wave-eq-var} and \eqref{smooth-1} we obtain that
$z(t)$ solves the problem
\begin{equation}\label{smooth-2}
(-\Delta+\mu)z= f_1(t)\; \mbox{ in }\; \Om,\quad
\frac{\partial z}{\partial\nu}= f_2(t)\; \mbox{ on }\; \Ga,
\end{equation}
where $f_1\in L_\infty(\R,L_2(\Om))$ and $f_2\in L_\infty(\R,H^s(\Ga))$
for any $s<3/2$.
By the elliptic regularity theory (see, e.g.,\ \cite[Chap.~5]{Triebel78})
we conclude that $z(t)$ is a bounded function with values in $H^2(\Om)$.

In the case $\{n=3\,,\, 3< p\le 5\}$ we have that $g(z_t)$ is bounded in
$L_{6/p}(\Om)$ and therefore $z$ solves \eqref{smooth-2} with
$f_1(t)\in L_\infty(\R; L_{6/p}(\Om))$.
Again, the elliptic regularity theory ensures that $z(t)$
is bounded in $W^2_{6/p}(\Om)$.
Thus, the second statement of Theorem~\ref{t:main2} is proved.
\qed

\medskip
{\em Proof of Theorem~\ref{co:up-sc}.}
The argument is standard (see, e.g., \cite{babin,chueshov-book} and
\cite{raquel2}) and relies on the uniform estimate
\eqref{regularity-estimate} of Theorem~\ref{t:main2}.
We proceed by contradiction. Assume that \eqref{semi-cont-at1} does not hold.
Then there exists a sequence $\la_n=(\ga_n,\ka_n)$ such that
$\la_n\to \la_0\equiv(\ga_0,\ka_0)$ and a sequence
$U_n\in \mfA^{\ga_n,\ka_n}$ such that
\begin{equation}\label{cntr-s}
{\rm dist}_{Y_{\ga_0}}\left( U_n,\mfA^{\ga_0,\ka_0}\right)\ge\delta>0\,,
\quad n=1,2,\ldots
\end{equation}
Let $\left\{ U_n(t)\, :\, t\in\R\right\}$ be a full trajectory from the
attractor $\mfA^{\ga_n,\ka_n}$ such that $U_n(0)= U_n$.
Since $W^2_{\min\{2,6/p\}}(\Om)$ is compactly
embedded in $H^1(\Om)$ for $p<5$,
using \eqref{first-bnd}, \eqref{regularity-estimate} and Aubin's compactness
theorem (see, e.g., \cite[Corollary~4]{simon}), we can conclude that there exists a sequence
$\{n_k\}$ and a function $U(t)\in C_b(\R,Y_{\ga_0})$ such that
\[
\max_{t\in [-T,T]}
\|U_{n_k}(t)-U(t)\|_{Y_{\ga_0}}\to 0\quad\mbox{as}\quad k\to\infty\,.
\]
Since $\|U(t)\|_{Y_{\ga_0}}\le R$ for all $t\in\R$ and  some $R>0$, the trajectory
$\left\{ U(t)\, :\, t\in\R\right\}$ belongs to the attractor $\mfA^{\ga_0}$.
Consequently,  $ U_{n_k}\to U(0)\in \mfA^{\ga_0,\ka_0}$,
which contradicts \eqref{cntr-s}.
\qed

\begin{remark}\label{re:up-sc2}
\begin{rm}
Since $W^2_{6/5}(\Om)$ is \emph{not} compactly
embedded in $H^1(\Om)$, the above proof fails in the case $p=5$.
Yet, we may use the compactness of the embedding
$W^2_{6/5}(\Om)\subset H^{1-\eps}(\Om)$ for arbitrary $\eps>0$
so as to establish the semi-continuity property in the space
$Y^{-\eps}_{\gamma_0}$ defined in \eqref{Y-g-eps}.
\end{rm}
\end{remark}

\appendix

\section{Sketch of the proof of Theorem~\ref{thm:well-posed} (well-posedness)}
\label{appendix}
Let us consider the abstract second-order system \eqref{abstract-system}
corresponding to the PDE model \eqref{pde-system}.
To give a first-order abstract formulation of \eqref{abstract-system}
in the variable $y=[z,z_t,v,v_t,\theta]^T$, we introduce the operator
$L:\cD(L)\subset Y\to Y$ defined by
\be
\label{dynamics-operator}
L \left[\!\! \ba{c} z_1 \\[1mm] z_2 \\[1mm] v_1 \\[1mm] v_2 \\[1mm] \theta  \ea \!\!\right] =
\left[\!\!\ba{c} -z_2  \\[1mm]  A z_1 + D(z_2)+z_2 - \alpha \kappa A N_0 v_2
\\[1mm]  - v_2
\\[1mm] M_{\gamma}^{-1}\left[ \cA^2 v_1
+ \beta\kappa N_0^* A z_2 + v_2-\cA\theta\right]
\\[1mm]    \cA\theta +\cA v_2
 \ea \!\!\right]\,,
\ee
with domain
\be\label{dom-t}
\ba{l}
\cD(L) = \Big\{
[z_1,z_2,v_1,v_2,\theta] \in \cD(A^{1/2})\times \cD(A^{1/2})
\times W_\gamma \times \cD(\cA)\times \cD(\cA):
\\[2mm]
\myspace\myspace
A[z_1 - \alpha \kappa N_0 v_2] + D(z_2)\in L_2(\Omega)\,\Big\}\,,
\ea
\ee
which is dense in $Y$.
Then, we see that system \eqref{abstract-system} can be rewritten as
\be\label{simple-system}
y' + L y = C(y), \quad y(0)=y_0\,,
\ee
with $L$ defined by \eqref{dynamics-operator} and $C$ given by
\[
C\left[z_1, z_2, v_1, v_2,\theta \right]
=
\left[0, -F_1(z_1)+z_2, 0,  -F_2(v_1)+v_2, 0\right]^T\,.
\]
It is not difficult to show that $L$ is a maximal monotone operator,
and hence the proof will be omitted; a detailed proof in the isothermal case
is found in \cite[Appendix~A]{bcl-cpaa-06}, see also \cite{bucci-03}
and \cite{las-cmbs}.
Moreover, as a consequence of \eqref{F}, the nonlinear term $C$ is locally
Lipschitz on the phase space $Y$.
Thus, the nonlinear equation \eqref{simple-system} is a locally
Lipschitz perturbation of a system driven by an $m$-monotone operator
and one may invoke \cite[Theorem~7.2]{chu-eller-las-1}, which yields
the local existence (and uniqueness) for both strong and
generalized solutions.
(The reader is referred to \cite{brezis-2,barbu,show} for classical results
on nonlinear semigroups and differential equations in Banach spaces).

The next step consists in establishing the energy inequality
\eqref{energy-ineq} and the identity \eqref{energy-identity} on the
solutions' existence interval.
Indeed, the equality \eqref{energy-identity} pertaining to strong solutions
is easily shown using a standard argument.
That generalized solutions satisfy the inequality \eqref{energy-ineq}
follows by a limit procedure, in view of \eqref{energy-identity}
and Definition~\ref{str-sol-2ord}(G).
Moreover, \eqref{energy-ineq} implies \eqref{energy-decr} and
\eqref{energy-est2}.
Finally, the latter estimate makes it possible to establish a global
existence result and therefore to conclude the proof.

\end{document}